\documentclass[reqno]{amsart}
\usepackage{amsmath,amssymb}
\usepackage{latexsym}
\usepackage[usenames, dvipsnames]{color}
\usepackage{pxfonts}
\usepackage[hypertex]{hyperref}
\usepackage{enumerate}

\usepackage{verbatim}

\allowdisplaybreaks[4]

\numberwithin{equation}{section}

\newtheorem{theorem}{Theorem}[section]
\newtheorem{corollary}[theorem]{Corollary}
\newtheorem{lemma}[theorem]{Lemma}

\theoremstyle{definition}
\newtheorem{remark}[theorem]{Remark}

\theoremstyle{definition}
\newtheorem{definition}[theorem]{Definition}

\theoremstyle{definition}
\newtheorem{assumption}[theorem]{Assumption}

\theoremstyle{definition}

\makeatletter
\def\dashint{\operatorname%
{\,\,\text{\bf-}\kern-.98em\DOTSI\intop\ilimits@\!\!}}
\makeatother

\def\\det{\text{det}}

\def\.5{\frac{1}{2}}

\def\cD{\mathcal{D}}

\def\cR{\mathcal{R}}

\newcommand{\RN}[1]{%
  \textup{\uppercase\expandafter{\romannumeral#1}}%
}

\newcommand{\Div}{\operatorname{div}}
\newcommand{\dist}{\text{dist}}

\newcounter{marnote}

\begin{document}


\title[Hessian estimates for non-divergence form elliptic equations]{Hessian estimates for non-divergence form elliptic equations arising from composite materials}

\author[H. Dong]{Hongjie Dong}
\address[H. Dong]{Division of Applied Mathematics, Brown University, 182 George Street, Providence, RI 02912, USA}
\email{Hongjie\_Dong@brown.edu }
\thanks{H. Dong was partially supported by the NSF under agreement DMS-1600593.}

\author[L. Xu]{Longjuan Xu}
\address[L. Xu]{School of Mathematical Sciences, Beijing Normal University, Laboratory of Mathematics and Complex Systems, Ministry of Education, Beijing 100875, China}
\address{\qquad\quad Division of Applied Mathematics, Brown University, 182 George Street, Providence, RI 02912, USA}
\email{ljxu@mail.bnu.edu.cn, longjuan\_xu@brown.edu}
\thanks{L. Xu was partially supported by the China Scholarship Council (No. 201706040139).}

\begin{abstract}
In this paper, we prove that any $W^{2,1}$ strong solution to second-order
non-divergence form elliptic equations is locally $W^{2,\infty}$ and piecewise $C^{2}$ when the leading coefficients and data are of
piecewise Dini mean oscillation and the lower-order terms are bounded. Somewhat surprisingly here the interfacial boundaries are only required to be $C^{1,\text{Dini}}$. We also derive global weak-type $(1,1)$ estimates with respect to $A_{1}$ Muckenhoupt weights. 
The corresponding results for the adjoint operator are
established. 
Our estimates are independent of the distance between these surfaces of discontinuity of the coefficients.
\end{abstract}

\maketitle

\section{Introduction and main results}

Let $\cD$ be a bounded domain in $\mathbb R^{n}$ that contains $M$ disjoint sub-domains $\cD_{1}, \ldots, \cD_{M}$ with $C^{1,\text{Dini}}$ boundaries,
that is, $\cD=(\cup_{j=1}^{M}\overline{\cD}_{j})\setminus\partial \cD$. For more details about $C^{1,\text{Dini}}$ boundaries, see Definition \ref{def Dini}. We suppose that if the boundaries of two $\cD_{j}$ touch, then they touch on a whole component of such a boundary.
We thus without loss of generality assume that $\partial\cD\subset\partial\cD_{M}$.

We consider the following second-order elliptic equation in non-divergence form
\begin{equation}\label{equations}
Lu:=a^{ij}D_{ij}u+b^{i}D_{i}u+cu=f
\end{equation}
in $\cD$,
where the Einstein summation convention on repeated indices is used. Throughout this paper, the coefficients $a^{ij}, b^{i}$, and $c$ are bounded by a positive constant $\Lambda$. We assume that the principal coefficients matrices $A=(a^{ij})_{i,j=1}^{n}$ are defined on $\mathbb R^{n}$ and uniformly elliptic with ellipticity constant $\delta\in (0,1)$:
$$\delta|\xi|^{2}\leq a^{ij}(x)\xi^{i}\xi^{j},\quad\forall~\xi=(\xi^{1},\ldots,\xi^{n})\in\mathbb R^{n},\quad\forall~x\in\mathbb R^{n}.$$
Without loss of generality, we may assume that $A$ is symmetric, i.e, $a^{ij}=a^{ji}$. We are interested in the case when the coefficients and data are allowed to be discontinuous across the interfacial boundaries. Such problem, in particular for the corresponding divergence form equations, has been studied by many authors. See, for instance, \cite{lv,ln,d,X13,XB13,ABEFV}.

In this paper, we prove the piecewise $C^{2}$ regularity and local $W^{2,\infty}$ estimate for $W^{2,1}$ strong solutions of \eqref{equations} when the coefficients and $f$ are piecewise Dini continuous in the $L^{1}$-mean sense in each subdomains. Moreover, when the sub-domains,  $\cD_{1},\dots,\cD_{M-1}$, are away from the boundary $\partial\cD$, we prove global weak-type $(1,1)$ estimates with $A_{1}$ Muckenhoupt weights for any $W^{2,1}$  strong solution of \eqref{equation} without imposing further conditions on the leading coefficients $a^{ij}$ other than being piecewise Dini mean oscillation over an open set containing $\overline{\cD}$; see Theorem \ref{thm weak C2}.

Our argument is based on Campanato's approach presented in \cite{c,g}, the key point of which is to show that the mean oscillation of $D^{2}u$ (or $Du$, or $u$, respectively) in balls vanishes in certain order as the radii of balls go to zero. The method was used recently for divergence and non-divergence form elliptic equations with coefficients satisfying certain conditions. For instance, in \cite{XB13} the authors derived very general BMO, Dini, H\"older, and higher regularity estimates for weak solutions to the corresponding divergence form systems by using Camapato's approach, where the estimates may depend on the distance between sub-domains. See also \cite{d} in which both divergence form systems and non-divergence form equations were studied when the subdomains are laminar.
In \cite{dk}, the authors studied $C^{2}$ and weak-type $(1,1)$ estimates of the solution to
\begin{equation}\label{equation}
a^{ij}D_{ij}u=f.
\end{equation}
They showed that any $W^{2,2}$ strong solution to \eqref{equation} is $C^{2}$ provided that the modulus of continuity of coefficients in the $L^{1}$-mean sense satisfies the Dini condition. Later, the authors in \cite{dek} extended and improved the results in \cite{dk}, by showing that any strong solution to elliptic equations in non-divergence form with zero Dirichlet boundary conditions is $C^2$ up to the boundary when the coefficients satisfy the same condition. The main obstacle in \cite{dek,dk} is that the usual argument based on $L^{p}~(p>1)$ estimates does not work because only the assumptions on the $L^{1}$-mean oscillations of the coefficients and data are imposed. To overcome it, they used weak-type $(1,1)$ estimates and adapted Campanato's method in the $L^{p}$ setting with $p\in (0,1)$. The above idea was also used in a recent paper \cite{dx}, where the authors showed that $W^{1,p},1\leq p<\infty,$  weak solutions to divergence form elliptic {\em systems} are Lipschitz and piecewise $C^{1}$ under the same conditions on the coefficients and data as imposed before. Hence, this paper can be regarded as a companion paper of \cite{dx}.

Similar to \cite{dx}, an added difficulty is the lack of regularity of $D^{2}u$ in one direction. For this, we adapt the scheme in \cite{dx} to our case. We point out that the coordinate system in our setting and \cite{dx} is chosen according to the geometry of the sub-domains and is different at each point. This is in contrast to \cite{d,dek,dk}, where the coordinate system is fixed. Therefore, our mean oscillation estimates depend on the balls under consideration, which makes the argument much more involved.

Denote by $\mathcal{A}$ the set of piecewise constant functions in each $\cD_{j}$, $j=1,\ldots,M$. We assume that $A$ is piecewise Dini continuous in the $L^1$ sense in $\cD$, that is,
\begin{align}\label{omega A}
\omega_{A}(r):=\sup_{x_{0}\in \cD}\inf_{\hat{A}\in\mathcal{A}}\fint_{B_{r}(x_{0})}|A(x)-\hat{A}|\ dx
\end{align}
satisfies the Dini condition, where $B_{r}(x_{0})\subset \cD$. For more details about the Dini condition, see Definition \ref{piece Dini}. For $\varepsilon>0$ small, we set
$$\cD_{\varepsilon}:=\{x\in \cD: \dist(x,\partial \cD)>\varepsilon\}.$$
Denote $b:=(b^{1},\ldots,b^{n})$.

Our first result reads that if the coefficients and $f$ are piecewise Dini continuous in the $L^1$ sense, then any $W^{2,p}$ strong solution to the above equation \eqref{equations} is locally $W^{2,\infty}$ and piecewise $C^{2}$.
\begin{theorem}\label{thm}
Let $\cD$ be defined as above. Let $\varepsilon\in (0,1)$, $p\in(1,\infty)$, and $\gamma\in(0,1)$. Assume that $A$, $b$, $c$, and $f$ are of piecewise Dini mean oscillation in $\cD$, and $f\in L^{\infty}(\cD)$. If $u\in W^{2,p}(\cD)$ is a strong solution to \eqref{equations} in $\cD$, then $u\in C^{2}(\overline{{\cD}_{j}\cap \cD_{\varepsilon}})$, $j=1,\ldots,M$, and $Du$ is Lipschitz in $\cD_{\varepsilon}$. Moreover, for any fixed $x\in \cD_{\varepsilon}$, there exists a coordinate system associated with $x$, such that for all $y\in \cD_{\varepsilon}$, we have
\begin{align}\label{main result}
&|D_{xx'}u(x)-D_{xx'}u(y)|\nonumber\\
&\leq C\int_{0}^{|x-y|}\frac{\tilde{\omega}_{A}(t)}{t}dt\cdot\Bigg(\|D^{2}u\|_{L^{1}(\cD)}+\int_{0}^{1}\frac{\tilde{\omega}_{f}(t)}{t}dt+\|f\|_{L^{\infty}(\cD)}+\|u\|_{L^{1}(\cD)}\Bigg)\nonumber\\
&\quad+C|x-y|^{\gamma}\Bigg(\|D^{2}u\|_{L^{1}(\cD)}+\|f\|_{L^{\infty}(\cD)}+\|u\|_{L^{1}(\cD)}\Bigg)+C\int_{0}^{|x-y|}\frac{\tilde{\omega}_{f}(t)}{t}dt,
\end{align}
where $C$ depends on $n,M,p,\delta,\Lambda,\varepsilon,\omega_{b},\omega_{c}$, and the $C^{1,\text{Dini}}$ characteristics of $\partial\cD_{j}$, $\tilde\omega_{\bullet}(t)$ is a Dini function derived from $\omega_{\bullet}(t)$; see \eqref{tilde phi}.
\end{theorem}

We note that in the above theorem the interfacial boundaries are only required to be in $C^{1,\text{Dini}}$, which is the same condition as in \cite{dx}. This is in contrast to the usual Dirichlet boundary value problem in which case for the $C^2$ estimate the boundary of the domain is assumed to be in $C^{2,\text{Dini}}$. See, for example, \cite{dek}.

Under the stronger condition that the coefficients and $f$ are piecewise H\"{o}lder continuous in $\cD$, we further show that $D_{xx'}u$ is H\"{o}lder continuous.
\begin{corollary}\label{coro}
Let $\cD$ be defined as above and each sub-domain has $C^{1,\mu}$ boundary with $\mu\in(0,1]$. Let $\varepsilon\in (0,1)$ and $p\in(1,\infty)$. Assume that $A, b, c, f\in C^{\alpha}(\overline{\cD}_{j})$ with $\alpha\in \Big(0,\mu/(1+\mu)\Big]$. If $u\in W^{2,p}(\cD)$ is a strong solution to \eqref{equations} in $\cD$, then the assertions of Theorem \ref{thm} also hold true. Furthermore, \eqref{main result} is replaced by
\begin{align}\label{DD holder}
|D_{xx'}u(x)-D_{xx'}u(y)|\leq C|x-y|^{\alpha}\left(\sum_{j=1}^{M}|f|_{\alpha;\overline{\cD}_{j}}+\|D^{2}u\|_{L^{1}(\cD)}+\|u\|_{L^{1}(\cD)}\right),
\end{align}
where $C$ depends on $n,M,\alpha,\mu,\delta,\Lambda,\varepsilon,p,\|A\|_{C^{\alpha}(\overline{\cD}_{j})},\|b\|_{C^{\alpha}(\overline{\cD}_{j})}$, $\|c\|_{C^{\alpha}(\overline{\cD}_{j})}$, and the $C^{1,\mu}$ norms of $\partial\cD_{j}$.
\end{corollary}

\begin{remark}
It follows from \eqref{DD holder} that $D^2 u\in C^{\alpha}\big({\overline{\cD}_{j}\cap \cD_{\varepsilon}}\big)$. Indeed, from the proof of Theorem \ref{thm}, we have $D^{2}u\in L_{\text{loc}}^{\infty}$ with
\begin{align*}
\|D^{2}u\|_{L^{\infty}(\cD_{\varepsilon})}\leq C\|D^2u\|_{L^{1}(\cD)}+C
\sum_{j=1}^{M}|f|_{\alpha;\overline{\cD}_{j}}+C\|u\|_{L^{1}(\cD)}.
\end{align*}
Then, $Du$ and $u$ are Lipschitz in $\cD_{\varepsilon}$. Since
$$D_{nn}u=\frac{1}{a^{nn}}\left(f-b^{i}D_{i}u-cu-\sum_{(i,j)\neq(n,n)}a^{ij}D_{ij}u\right),$$
we also have $D_{nn}u\in C^{\alpha}\big({\overline{\cD}_{j}\cap \cD_{\varepsilon}}\big)$. Therefore, $u\in C^{2,\alpha}\big({\overline{\cD}_{j}\cap \cD_{\varepsilon}}\big)$. By using interpolation inequalities, the term $\|D^2 u\|_{L^1(\cD)}$ on the right-hand side of \eqref{main result} and \eqref{DD holder} can be dropped. We also point out that for $L^p$-viscosity solutions, a result similar to Corollary \ref{coro} was obtained in \cite{X13} when $\alpha\in \Big(0,\mu/(n(1+\mu))\Big]$ by using a different argument.
\end{remark}

To prove the above results, we need to consider the formal adjoint operator defined by
$$L^{*}u=D_{ij}(a^{ij}u)-D_{i}(b^{i}u)+cu,$$
and deal with the following boundary value problem
\begin{align}\label{adjoint}
\begin{cases}
L^{*}u=\Div^{2} g&\quad\mbox{in}~\cD,\\
u=\frac{g\nu\cdot\nu}{A\nu\cdot\nu}&\quad\mbox{on}~\partial\cD,
\end{cases}
\end{align}
where $g=(g^{ij})_{i,j=1}^{n}$, $\Div^{2} g=D_{ij}g^{ij}$ with $g\in L^{\infty}(\cD)$, and $\nu$ is the unit outer normal vector on $\partial\cD$. For more details about the adjoint solution to \eqref{adjoint}, see Definition \ref{def adjoint}. By using a similar idea to that in the proof of Theorem \ref{thm}, we also obtain the corresponding results for the adjoint problem \eqref{adjoint}.

\begin{theorem}\label{thm C0}
Let $\cD$ be defined as in Theorem \ref{thm}. Let $\varepsilon\in (0,1)$, $p\in(1,\infty)$, and $\gamma\in(0,1)$. Suppose that $A$, $b$, $c$, and $g$ are of piecewise Dini mean oscillation in $\cD$, $g\in L^{\infty}(\cD)$. Let $u\in L^{p}(\cD)$ be a local adjoint solution of
$$L^{*}u=\Div^{2} g\quad\mbox{in}~\cD.$$
Then $u\in L^{\infty}(\cD_{\varepsilon})$. Moreover, for any fixed $x\in \cD_{\varepsilon}$, there exists a coordinate system associated with $x$, such that for any $y\in \cD_{\varepsilon}$, we have
\begin{align}\label{modulus u}
&|\bar{u}(x)-\bar{u}(y)|\nonumber\\
&\leq C\int_{0}^{|x-y|}\frac{\tilde{\omega}_{A}(t)}{t}\ dt\cdot\Bigg(\int_{0}^{1}\frac{\tilde{\omega}_{g}(t)}{t}\ dt+\|g\|_{L^{\infty}(\cD)}+\|u\|_{L^{p}(\cD)}\Bigg)\nonumber\\
&\quad+C|x-y|^{\gamma}\Bigg(\|g\|_{L^{\infty}(\cD)}+\|u\|_{L^{p}(\cD)}\Bigg)+C\int_{0}^{|x-y|}\frac{\tilde{\omega}_{g}(t)}{t}\ dt,
\end{align}
where $\bar{u}=a^{nn}u-g^{nn}$
and $C$ depends on $n$, $M$, $p$,  $\delta$, $\Lambda$, $\varepsilon$, $\omega_{b}$, $\omega_{c}$, and the $C^{1,\text{Dini}}$ characteristics of $\partial\cD_{j}$.
\end{theorem}

\begin{corollary}\label{coro u}
Let $\cD$ be defined as in Corollary \ref{coro}. Let $\varepsilon\in (0,1)$ and $p\in(1,\infty)$. Assume that $A, b, c, g\in C^{\alpha}(\overline{\cD}_{j})$ with $\alpha\in (0,1)$. 
If $u\in L^{p}(\cD)$ is a local adjoint solution of
$$L^{*}u=\Div^{2} g\quad\mbox{in}~\cD.$$
Then the assertions of Theorem \ref{thm C0} also hold true, and \eqref{modulus u} is replaced with
\begin{align*}
|\bar{u}(x)-\bar{u}(y)|\leq C|x-y|^{\alpha}\left(\sum_{j=1}^{M}|g|_{\alpha;\overline{\cD}_{j}}+\|u\|_{L^{p}(\cD)}\right),
\end{align*}
where $C$ depends on $n,M,p,\alpha,\mu,\delta,\Lambda,\varepsilon,\|A\|_{C^{\alpha}(\overline{\cD}_{j})}, \|b\|_{C^{\alpha}(\overline{\cD}_{j})}$, $\|c\|_{C^{\alpha}(\overline{\cD}_{j})}$, and the $C^{1,\mu}$ norms of $\partial\cD_{j}$.
\end{corollary}

\begin{remark}
Restricting to each $\cD_{j}\cap \cD_{\varepsilon}$, since
$u=(a^{nn})^{-1}\big(\bar{u}+g^{nn}\big)$,
$a^{nn}$ and $g^{nn}$ are in $C^{\alpha}(\overline{\cD}_{j})$, we conclude that $u\in C^{\alpha}(\overline{\cD}_{j}\cap \cD_{\varepsilon})$.
\end{remark}

By using a duality argument and Theorem \ref{thm C0}, we derive the following corollary.

\begin{corollary}\label{thm W21}
Let $A$, $b$, $c$, and $f$ be as in Theorem \ref{thm}. If $u\in W^{2,1}(\cD)$ is a strong solution to \eqref{equations} in $\cD$, then we have $u\in W^{2,p}_{\text{loc}}(\cD)$ for some $p\in (1,\infty)$, and the conclusion of Theorem \ref{thm} still holds.
\end{corollary}

We refer the reader to \cite{em} for a related result.

\begin{remark}
Similar to Corollary \ref{thm W21}, we can show that under the assumptions imposed in Theorem \ref{thm C0}, if $u\in L^{1}(\cD)$ is a local adjoint solution of
$$L^{*}u=\Div^{2} g\quad\mbox{in}~\cD,$$
then $u\in L^{p}_{\text{loc}}(\cD)$ for some $p\in (1,\infty)$, and the conclusion of Theorem \ref{thm C0} still holds true.
\end{remark}

Throughout this paper, unless otherwise stated, $C$ denotes a constant independent of the distance between sub-domains. 

The rest of this paper is organized as follows. In Section \ref{preliminaries}, we fix our domain and the coordinate system. We also introduce some notation, definitions, and auxiliary results used in this paper. In Section \ref{proof thm}, we provide the proofs of Theorem \ref{thm} and Corollary \ref{coro}. We prove Theorem \ref{thm C0} and Corollary \ref{coro u} in Section \ref{sec thm C0} and Corollary \ref{thm W21} in Section \ref{pf of thm W21}. We give a global weak-type $(1,1)$ estimate in Section \ref{sec thm weak}. In the appendix, we give a weighted $W^{2,p}$-estimate and solvability for non-divergence form elliptic equations in $C^{1,1}$ domains with the zero Dirichlet boundary condition, which is of independent interest.

\section{Preliminaries}\label{preliminaries}

In this section, we fix our domain and list some notation, definitions, and auxiliary lemmas used in the paper.

\subsection{Notation and definitions}

We follow the notation and definitions from \cite{dx}. Write $x=(x^{1},\dots,x^{n})=(x',x^{n})$, $n\geq2$. Denote
$$B_{r}(x):=\{y\in\mathbb R^{n}: |y-x|<r\},\quad B'_{r}(x'):=\{y'\in\mathbb R^{n-1}: |y'-x'|<r\},$$
and $B_{r}:=B_{r}(0)$, $B'_{r}:=B'_{r}(0')$, $\cD_{r}(x):=\cD\cap B_{r}(x)$. For a function $f$ defined in $\mathbb R^{n}$, we set
$$(f)_{\cD}=\frac{1}{|\cD|}\int_{\cD}f(x)\ dx=\fint_{\cD}f(x)\ dx,$$
where $|\cD|$ is the $n$-dimensional Lebesgue measure of $\cD$. We shall use the notation
$$[u]_{k;\cD}:=\sup_{x\in\cD}|D^{k}u(x)|\quad\mbox{and}\quad[u]_{k,\gamma;\cD}:=\sup_{\begin{subarray}{1}x,y\in\cD\\
x\neq y
\end{subarray}}\frac{|D^{k}u(x)-D^{k}u(y)|}{|x-y|^{\gamma}},$$
where $k=0,1,\ldots,$ and $\gamma\in(0,1]$. For $k=0$, we denote $[u]_{\gamma;\cD}:=[u]_{0,\gamma;\cD}$ for abbreviation. We also define
$$|u|_{k;\cD}:=\sum_{j=0}^{k}[u]_{j;\cD}\quad\mbox{and}\quad|u|_{k,\gamma;\cD}:=|u|_{k;\cD}+[u]_{k,\gamma;\cD}.$$
We denote $C^{k,\gamma}(\cD)$ to be the set of bounded measurable functions $u$ that are $k$-times continuously differentiable in $\cD$ and $[u]_{k,\gamma;\cD}<\infty$. Moreover, the following notation will be used:
$$D_{x'}u=u_{x'},~ D_{x'}^{2}u=u_{x'x'},~ D_{xx'}u=u_{xx'},~ DD_{x'}^{2}u=u_{xx'x'},~ D^{2}D_{x'}u=u_{xxx'}.$$

For a function $f$, $k=1,2,\ldots,n$, and $h>0$, we define the finite difference quotient
$$
\delta_{h,k}f(x):=\frac{f(x+he_{k})-f(x)}{h}.
$$

\begin{definition}\label{piece Dini}
We say that a continuous increasing function $\omega: [0,1]\rightarrow\mathbb R$ satisfies the Dini condition if $\omega(0)=0$ and
$$\int_{0}^{t}\frac{\omega(s)}{s}\ ds<+\infty,\quad\forall~t\in(0,1].$$
\end{definition}

\begin{definition}\label{def Dini}
Let $k\ge 1$ be an integer and $\cD\subset\mathbb R^{n}$ be open and bounded. We say that $\partial\cD$ is $C^{k,\text{Dini}}$ if for each point $x_{0}\in\partial\cD$, there exists $R_{0}\in(0,1/8)$ independent of $x_{0}$ and a $C^{k,\text{Dini}}$ function (that is, $C^{k}$ function whose $k$-th derivatives are Dini continuous) $\varphi: B'_{R_{0}}\rightarrow\mathbb R$ such that (upon relabeling and reorienting the coordinates axes if necessary) in a new coordinate system $(x',x^{n})$, $x_{0}$ becomes the origin and
$$\cD_{R_{0}}(0)=\{x\in B_{R_{0}}: x^{n}>\varphi(x')\},\quad\varphi(0')=0,$$
and $D_{x'}^k\varphi$ has a modulus of continuity $\omega_0$, a Dini function which is increasing, concave, and independent of $x_0$.
\end{definition}

\begin{definition}\label{def weight}
(1) We say $w: \mathbb R^{n}\rightarrow[0,\infty)$ belongs to $A_{1}$ if there exists some constant $C$ such that for all balls $B$ in $\mathbb R^{n}$,
$$\fint_{B}w(y)\ dy\leq C\inf_{x\in B}w(x).$$
The $A_1$ constant $[w]_{A_1}$ of $w$ is defined as the infimum of all such $C$'s.

(2) We say $w: \mathbb R^{n}\rightarrow[0,\infty)$ belongs to $A_{p}$ for $p\in(1,\infty)$ if
$$\sup_{B}\frac{w(B)}{|B|}\left(\frac{w^{\frac{-1}{p-1}}(B)}{|B|}\right)^{p-1}<\infty,$$
where the supremum is taken over all balls in $\mathbb R^{n}$. The value of the supremum is the $A_{p}$ constant of $w$, and will be denoted by $[w]_{A_p}$.
\end{definition}

The following definition is extracted from \cite{em}.

\begin{definition}\label{def adjoint}
Let $g\in L^{p}(B_{1})$, $1<p<\infty$ and $1/p+1/p'=1$. We say that $u\in L^{p}(B_{1})$ is an adjoint solution of \eqref{adjoint} if $u$ satisfies
\begin{align}\label{iden adjoint}
\int_{B_{1}}uLv=\int_{B_{1}}\mbox{tr}(gD^{2}v),
\end{align}
for any $v\in W^{2,p'}(B_{1})\cap W_{0}^{1,p'}(B_{1})$, where $\mbox{tr}(gD^{2}v)=g^{ij}D_{ij}v$. By a local adjoint solution of
$$L^{*}u=\Div^{2} g\quad\mbox{in}~B_{1},$$
we mean a function $u\in L_{\text{loc}}^{p}(B_{1})$ that verifies \eqref{iden adjoint} for any $v\in W_{0}^{2,p'}(B_{1})$.
\end{definition}

\subsection{Some properties of the domain, coefficients, and data}\label{subsection domain}

Below, we slightly abuse the notation. Let $\cD$ be the unit ball $B_{1}$ and take $x_{0}\in B_{3/4}$. We now localize and fix our domain as follows. By suitable rotation and scaling, we may suppose that a finite number of sub-domains lie in $B_{1}$ and that they take the form
$$x^{n}=h_{j}(x'),\quad\forall~x'\in B'_{1},~j=1,\ldots,l\ (\le M),$$
with
$$-1<h_{1}(x')<\dots<h_{l}(x')<1,$$
and $h_{j}(x')\in C^{1,\text{Dini}}(B'_{1})$. Set $h_{0}(x')\equiv -1$ and $h_{l+1}(x')\equiv 1$ so that we have $l+1$ regions:
$$\cD_{j}:=\{x\in \cD: h_{j-1}(x')<x^{n}<h_{j}(x')\},\quad1\leq j\leq l+1.$$ We may suppose that there exists some $\cD_{j_{0}}$, such that $x_{0}\in B_{3/4}\cap \cD_{j_{0}}$ and the closest point on $\partial \cD_{j_{0}}$ to $x_{0}$ is $(x'_{0},h_{j_{0}}(x'_{0}))$, and $\nabla_{x'}h_{j_{0}}(x'_{0})=0'$ after a proper rotation.
We introduce the $l+1$ ``strips"
$$\Omega_{j}:=\{x\in \cD: h_{j-1}(x'_{0})<x^{n}<h_{j}(x'_{0})\},\quad1\leq j\leq l+1.$$
Then we have the following result, which is \cite[Lemma 2.3]{dx}.
\begin{lemma}\label{volume}
There exists a constant $C$, depending on $n,l$ and the $C^{1,\text{Dini}}$ characteristics of $h_{j}$, $1\leq j\leq l$, such that
$$r^{-n}|(\cD_{j}\Delta\Omega_{j})\cap B_{r}(x_{0})|\leq C\omega_{1}(r),\quad1\leq j\leq l+1,~0<r<r_0:=\frac 2 3\int_0^{R_0/2}\omega_0'(s)s\,ds,$$
where $\cD_{j}\Delta\Omega_{j}=(\cD_{j}\setminus\Omega_{j})\cup(\Omega_{j}\setminus \cD_{j})$, $R_{0}$ is defined in Definition \ref{def Dini}, $\omega_0'$ denotes the left derivative of $\omega_0$, and $\omega_{1}(r):=\omega_{0}(2r+R)$ is a Dini function for some constant $R:=R(r)>2r$ satisfying
\begin{equation*}
\int_{0}^{R}\omega_0'(2r+s)s\ ds=3r/2.
\end{equation*}
\end{lemma}

Let $\hat{A}^{(j)}\in\mathcal{A}$ be a constant function in $\cD_{j}$ which corresponds to the definition of $\omega_{A}(r)$ in \eqref{omega A}. We define piecewise constant (matrix-valued) functions
\begin{align*}
\bar{A}(x)=
\hat{A}^{(j)},\quad x\in\Omega_{j}.
\end{align*}
Using $\hat{b}^{(j)}$ and $\hat{f}^{(j)}$, which are also constant functions in $\cD_{j}$, we similarly define piecewise constant functions $\bar{b}$ and $\bar{f}$. From Lemma \ref{volume} and the boundedness of $A$, we have
\begin{align}\label{est A}
\fint_{B_{r}(x_{0})}|\hat{A}-\bar{A}|\ dx
\leq C(n,\Lambda)r^{-n}\sum_{j=1}^{l+1}|(\cD_{j}\Delta\Omega_{j})\cap B_{r}(x_{0})|\leq C\omega_{1}(r),
\end{align}
which is also true for $\hat{b}$ and $\hat{f}$.

\subsection{Some auxiliary lemmas}\label{sub auxiliary}

We first recall the $W^{2,p}$-solvability for elliptic equations with leading coefficients which are variably partially VMO (vanishing mean oscillation) in the interior of $B_{1}$ and VMO near the boundary. We choose a cut-off function $\eta\in C_{0}^{\infty}(B_{7/8})$ with
$$0\leq\eta\leq1,\quad \eta\equiv1~\mbox{in}~B_{3/4},\quad|\nabla\eta|\leq 16.$$
Let $\tilde{L}$ be the elliptic operator defined by
\begin{align*}
\tilde{L}u=\tilde{a}^{ij}D_{ij}u+b^{i}D_{i}u+cu,
\end{align*}
where $\tilde{a}^{ij}=\eta a^{ij}+\delta(1-\eta)\delta_{ij}$, $\delta$ is the ellipticity constant of $a^{ij}$, and $\delta_{ij}$ is the Kronecker delta symbol. Consider
\begin{align}\label{approxi sol}
\begin{cases}
\tilde{L}u-\lambda u=f&\quad\mbox{in}~B_{1},\\
u=0&\quad\mbox{on}~\partial B_{1},
\end{cases}
\end{align}
where $\lambda\geq0$ is a constant, $f\in L^{p}(B_{1})$. Here $\tilde{a}^{ij}$ satisfy the conditions of \cite[Theorem 2.5]{d1} in the interior of the domain $B_{1}$, i.e., for a sufficiently small constant $\gamma_{0}=\gamma_{0}(n,p,\delta)\in(0,1)$ to be specified later,
we can find a constant $r_{0}\in(0,1)$ such that the following holds: For any $x_{0}\in B_{1}$ and $r\in(0,\min\{r_{0},\mbox{dist}(x_{0},\partial B_{1})/2\}]$ (so that $B_{r}(x_{0})\subset B_{1}$), in a coordinate system depending on $x_{0}$ and $r$, one can find a symmetric $\bar{a}^{ij}:=\bar{a}^{ij}(x^{n})$ satisfying the ellipticity condition and
\begin{equation}\label{condi A}
\fint_{B_{r}(x_{0})}|\tilde{a}^{ij}(x)-\bar{a}^{ij}|\ dx\leq\gamma_{0}.
\end{equation}
Moreover, for any $x_{0}\in\partial B_{1}$ and $r\in(0,r_{0}]$, we have
$$
\fint_{B_{r}(x_{0})}|\tilde{a}^{ij}(x)-(\tilde{a}^{ij})_{B_{r}(x_{0})}|\ dx\leq\gamma_{0}.
$$
Applying \cite[Theorem 2.5]{d1} to our case and using the argument in \cite[Sections 8.5 and 11.3]{k}, we get

\begin{lemma}\label{solvability}
For any $p\in(1,\infty)$, there exists a small constant $\gamma_{0}=\gamma_{0}(n,p,\delta)\in(0,1)$ such that the following hold.
\begin{enumerate}
\item
For any $u\in W^{2,p}(B_{1})\cap W_{0}^{1,p}(B_{1})$,
$$\lambda\|u\|_{L^{p}(B_{1})}+\sqrt{\lambda}\|Du\|_{L^{p}(B_{1})}+\|D^{2}u\|_{L^{p}(B_{1})}\leq C\|\tilde{L}u-\lambda u\|_{L^{p}(B_{1})},$$
provided that $\lambda\geq\lambda_{0}$, where $C$ and $\lambda_{0}\geq0$ depend on $n,p,\delta$, $\Lambda$, and $r_0$.

\item
For any $\lambda>\lambda_{0}$ and $f\in L^{p}(B_{1})$, \eqref{approxi sol} admits a unique solution $u\in W^{2,p}(B_{1})\cap W_{0}^{1,p}(B_{1})$. In the case when $c\le 0$, we can take $\lambda=0$.
\end{enumerate}
\end{lemma}


In the Appendix, we shall prove a more general result in weighted Sobolev spaces. From \cite[Theorem 2.5]{d1}, the Sobolev embedding theorem, and a standard localization argument which is similar to that in the proof of \cite[Lemma 4]{d}, we also have the following interior estimates.
\begin{lemma}\label{nondiv lemma}
Let $p\in(1,\infty)$. there exists a small constant $\gamma_{0}=\gamma_{0}(n,p,\delta)\in(0,1)$ such that the following hold. Assume that $u\in W_{\text{loc}}^{2,p}$ satisfies $a^{ij}D_{ij}u+b^{i}D_{i}u+cu=f$ in $B_{1}$, where $f\in L^{p}(B_{1})$. Then there exists a constant $C=C(n,\delta,\Lambda,p,r_{0})$ such that
$$\|u\|_{W^{2,p}(B_{1/2})}\leq C\big(\|u\|_{L^{1}(B_{1})}+\|f\|_{L^{p}(B_{1})}\big).$$
In particular, if $p>n$, it holds that
$$|u|_{1,\gamma;B_{1/2}}\leq C\big(\|u\|_{L^{1}(B_{1})}+\|f\|_{L^{p}(B_{1})}\big),$$
where $\gamma=1-n/p$ and $C$ depends only on $n$, $p$, $\delta$, $\Lambda$, and $r_0$.
\end{lemma}

The adjoint operator corresponding to $\tilde{L}$ is defined by
$$\tilde{L}^{*}u:=D_{ij}(\tilde{a}^{ij}u)-D_{i}(b^{i}u)+cu.$$
\begin{lemma}\label{sol adjoint}
Let $q\in(1,\infty)$. Assume that $g=(g^{ij})_{i,j=1}^{n}\in L^{q}(B_{1})$. Then, there is a $\lambda_{0}\geq0$ depending on $n,q,\delta,\Lambda$, and $r_{0}$, such that for any $\lambda>\lambda_{0}$,
\begin{align}\label{adjoint sol}
\begin{cases}
\tilde{L}^{*}u-\lambda u=\Div^{2}g&\quad\mbox{in}~B_{1},\\
u=\frac{g\nu\cdot\nu}{A\nu\cdot\nu}&\quad\mbox{on}~\partial B_{1}
\end{cases}
\end{align}
admits a unique adjoint solution $u\in L^{q}(B_{1})$. Moreover, the following estimate holds,
$$\|u\|_{L^{q}(B_{1})}\leq C\|g\|_{L^{q}(B_{1})},$$
where $C=C(n,q,\delta,\Lambda,r_{0})$. In the case when $ c\le 0$, we can take $\lambda=0$.
\end{lemma}

\begin{proof}
For $f\in L^{p}(B_{1})$ with ${1}/{p}+{1}/{q}=1$, it follows from Lemma \ref{solvability} that there exists a unique solution $v\in W^{2,p}(B_{1})\cap W_{0}^{1,p}(B_{1})$ such that $\tilde{L}v-\lambda v=f$ a.e. in $B_{1}$, provided that $\lambda>\lambda_{0}$. Moreover,
\begin{align}\label{est W2p}
\lambda\|v\|_{L^{p}(B_{1})}+\sqrt{\lambda}\|Dv\|_{L^{p}(B_{1})}+\|D^{2}v\|_{L^{p}(B_{1})}\leq C\|f\|_{L^{p}(B_{1})}.
\end{align}
Define the functional $T: L^{p}(B_{1})\rightarrow\mathbb R$ by
\begin{align}\label{def func}
T(f):=\int_{B_{1}}\mbox{tr}(gD^{2}v)\ dx.
\end{align}
Combining \eqref{est W2p}, \eqref{def func}, and H\"{o}lder's inequality, we have
\begin{align*}
|T(f)|\leq \|D^{2}v\|_{L^{p}(B_{1})}\|g\|_{L^{q}(B_{1})}\leq C\|f\|_{L^{p}(B_{1})}\|g\|_{L^{q}(B_{1})}.
\end{align*}
Hence, $T$ is a bounded functional on $L^{p}(B_{1})$. By the Riesz representation theorem, there exists a unique $u\in L^{q}(B_{1})$, such that
\begin{align}\label{riesz}
T(f)=\int_{B_{1}}uf\ dx,\quad\forall~f\in L^{p}(B_{1}).
\end{align}
Moreover,
$$\|u\|_{L^{q}(B_{1})}\leq C\|g\|_{L^{q}(B_{1})}.$$
It follows from \eqref{def func} and \eqref{riesz} that
$$\int_{B_{1}}u(\tilde{L}v-\lambda v)\ dx=\int_{B_{1}}\mbox{tr}(gD^{2}v)\ dx,$$
that is, $u\in L^{q}(B_{1})$ is the unique adjoint solution to \eqref{adjoint sol}.
\end{proof}

Now we denote
$$\bar{L}_{0}u:=\bar{a}^{ij}(x^{n})D_{ij}u,$$
where $\bar{a}^{ij}(x^{n})$ satisfies the same ellipticity and boundedness conditions as $a^{ij}(x)$. 
\begin{lemma}\label{lemma DDDx'}
Assume that $u\in C_{\text{loc}}^{1,1}$ satisfies $\bar{L}_{0}u=\bar{f}(x^{n})$ in $B_{1}$, where $\bar{f}\in L^{\infty}(B_{1})$. Then for any $p\in(0,\infty)$, there exists a constant $C=C(n,\delta,\Lambda,p)$ such that for any $\mathbf c\in\mathbb R^{(n-1)\times(n-1)}$,
\begin{align}\label{D2Dx' u}
\|D^{2}D_{x'}u\|_{L^{\infty}(B_{1/2})}\leq C\|D_{x'}^{2}u-\mathbf c\|_{L^{p}(B_{1})}.
\end{align}
\end{lemma}

\begin{proof}
By using the finite difference quotient technique and applying Lemma \ref{nondiv lemma} with a slightly smaller domain, one can see that $D_{x'}^{2}u-\mathbf c\in W^{2,q}_{\text{loc}}(B_1)$ satisfies
$$\bar{L}_{0}(D_{x'}^{2}u-\mathbf c)=0\quad\mbox{in}~B_{1}$$
and
$$\|D_{x'}^{2}u-\mathbf c\|_{W^{2,q}(B_{1/2})}\leq C\|D_{x'}^{2}u-\mathbf c\|_{L^{2}(B_{3/4})}.$$
If $q>n$, we obtain
\begin{align}\label{DDx'2}
\|D_{x'}^{2}u-\mathbf c\|_{L^{\infty}(B_{1/2})}+\|DD_{x'}^{2}u\|_{L^{\infty}(B_{1/2})}\leq C\|D_{x'}^{2}u-\mathbf c\|_{L^{2}(B_{3/4})}.
\end{align}
Because $\bar{L}_{0}(D_{x'}u)=0$ in $B_{1}$, one can see that
$$D_{nn}D_{x'}u=\frac{-1}{\bar{a}^{nn}(x^{n})}\sum_{(i,j)\neq(n,n)}\bar{a}^{ij}(x^{n})D_{ij}D_{x'}u.$$
Therefore, using \eqref{DDx'2} and $\bar{a}^{nn}(x^{n})\geq\delta$, we have
$$\|D_{nn}D_{x'}u\|_{L^{\infty}(B_{1/2})}\leq C\|D_{x'}^{2}u-\mathbf c\|_{L^{2}(B_{3/4})},$$
which implies that
\begin{align}\label{D2xDx'}
\|D^{2}D_{x'}u\|_{L^{\infty}(B_{1/2})}\leq C\|D_{x'}^{2}u-\mathbf c\|_{L^{2}(B_{3/4})}.
\end{align}
For any $0<p<1<\infty$, by using H\"{o}lder's inequality, we get
\begin{align}\label{D2x'}
\|D_{x'}^{2}u-\mathbf c\|_{L^{2}(B_{4/5})}\leq \|D_{x'}^{2}u-\mathbf c\|_{L^{p}(B_{4/5})}^{\frac{p}{2}}\|D_{x'}^{2}u-\mathbf c\|_{L^{\infty}(B_{4/5})}^{1-\frac{p}{2}}.
\end{align}
Combining \eqref{DDx'2}, \eqref{D2x'}, and H\"{o}lder's inequality, we obtain
\begin{align*}
\|D_{x'}^{2}u-\mathbf c\|_{L^{\infty}(B_{3/4})}&\leq C\|D_{x'}^{2}u-\mathbf c\|_{L^{p}(B_{4/5})}^{\frac{p}{2}}\|D_{x'}^{2}u-\mathbf c\|_{L^{\infty}(B_{4/5})}^{1-\frac{p}{2}}\\
&\leq\frac{1}{2}\|D_{x'}^{2}u-\mathbf c\|_{L^{\infty}(B_{4/5})}+C\|D_{x'}^{2}u-\mathbf c\|_{L^{p}(B_{4/5})}.
\end{align*}
By a well-known iteration argument (see, for instance, \cite[Lemma 3.1 of Ch. V]{g}), we get
\begin{align}\label{all p case}
\|D_{x'}^{2}u-\mathbf c\|_{L^{\infty}(B_{3/4})}\leq C\|D_{x'}^{2}u-\mathbf c\|_{L^{p}(B_{1})},\quad \forall~p>0.
\end{align}
Coming back to \eqref{D2xDx'}, we obtain \eqref{D2Dx' u}.
The lemma is thus proved.
\end{proof}

In the proof of Theorem \ref{thm C0}, we need to use the following
\begin{lemma}\label{div xn}
Assume $u\in C_{\text{loc}}^{0,1}$ satisfies
$$
L_{0}u:=D_{i}(\bar{a}^{ij}(x^{n})D_{j}u)=0
$$ in $B_1$.
Then for any $p\in(0,\infty)$, there exists a constant $C=C(n,p,\delta,\Lambda)$ such that for any constant $c\in\mathbb R$, we have
\begin{align}\label{DDu}
\|Du\|_{L^{\infty}(B_{1/2})}\leq C\|u-c\|_{L^{p}(B_{1})}.
\end{align}
\end{lemma}

\begin{proof}
We first assume that $c=0$. It directly follows from \cite[Lemma 2.5]{dx} that for any $q\in(1,\infty)$,
\begin{align}\label{est u-c}
\|u\|_{W^{1,q}(B_{4/5})}\leq C\|u\|_{L^{2}(B_{1})}.
\end{align}
Then for $q>n$ by the Sobolev embedding theorem, we have
\begin{align*}
\|u\|_{L^{\infty}(B_{4/5})}\leq C\|u\|_{L^{2}(B_{1})}.
\end{align*}
For $0<p<1$, by using a similar argument used in deriving \eqref{all p case}, we get
\begin{align}\label{u infinite}
\|u\|_{L^{\infty}(B_{4/5})}\leq C\|u\|_{L^{p}(B_{1})}.
\end{align}
For $k=1,\ldots,n-1$ and $h\in (0,1/12)$, since $\bar{a}^{ij}(x^{n})$ is independent of $x'$, we have $L_{0} (\delta_{h,k}u)=0$ in $B_{2/3}$. We thus use \cite[Lemma 2.5]{dx} again and \eqref{u infinite} to get
\begin{align*}
\|\delta_{h,k}u\|_{W^{1,q}(B_{1/2})}\leq C\|\delta_{h,k}u\|_{L^{2}(B_{2/3})}\leq C\|D_{x'}u\|_{L^{2}(B_{3/4})}
\leq C\|u\|_{L^{2}(B_{4/5})}\leq C\|u\|_{L^{p}(B_{1})},\quad p>0.
\end{align*}
Letting $h\rightarrow0$ gives
\begin{align}\label{Dx'u WW}
\|D_{x'}u\|_{W^{1,q}(B_{1/2})}\leq C\|u\|_{L^{p}(B_{1})},\quad p>0.
\end{align}
Moreover, notice that in $B_{1}$,
\begin{align*}
D_{n}U=-\sum_{i=1}^{n-1}\sum_{j=1}^{n}\bar{a}^{ij}D_{ij}u,\quad D_{x'}U=\sum_{j=1}^{n}\bar{a}^{nj}D_{x'}D_{j}u,
\end{align*}
where $U:=\bar{a}^{nj}(x^{n})D_{j}u$. Then by using \eqref{est u-c}, \eqref{u infinite}, \eqref{Dx'u WW}, and the boundedness of $\bar{a}^{nj}(x^{n})$, we obtain
\begin{align}\label{esti U}
\|U\|_{W^{1,q}(B_{1/2})}&=\|U\|_{L^{q}(B_{1/2})}+\|DU\|_{L^{q}(B_{1/2})}\nonumber\\
&\leq C\|Du\|_{L^{q}(B_{1/2})}+C\Big(\|D_{x'}U\|_{L^{q}(B_{1/2})}+\|D_{n}U\|_{L^{q}(B_{1/2})}\Big)\nonumber\\
&\leq C\|u\|_{W^{1,q}(B_{1/2})}+C\|DD_{x'}u\|_{L^{q}(B_{1/2})}\nonumber\\
&\leq C\|u\|_{L^{2}(B_{2/3})}+C\|D_{x'}u\|_{W^{1,q}(B_{1/2})}\leq C\|u\|_{L^{p}(B_{1})}.
\end{align}
Combining \eqref{Dx'u WW} and \eqref{esti U}, 
by the Sobolev embedding theorem for $q>n$, we have
\begin{equation*}
\|D_{x'}u\|_{L^{\infty}(B_{1/2})}+\|U\|_{L^{\infty}(B_{1/2})}\leq C\|u\|_{L^{p}(B_{1})},\quad p>0.
\end{equation*}
Thus, by $\bar{a}^{nn}(x^{n})\geq\delta$, we get
$$\|Du\|_{L^{\infty}(B_{1/2})}\leq C\|u\|_{L^{p}(B_{1})},\quad p>0.$$
Now, replacing $u$ with $u-c$, we conclude \eqref{DDu}.
\end{proof}

We will also apply the following lemma, which is \cite[Lemma 2.7]{dk}.

\begin{lemma}\label{lemma omiga}
Let $\omega$ be a nonnegative bounded function. Suppose there is $c_{1},c_{2}>0$ and $0<\kappa<1$ such that for $\kappa t\leq s\leq t$ and $0<t<r$,
\begin{align}\label{equivalence}
c_{1}\omega(t)\leq \omega(s)\leq c_{2}\omega(t).
\end{align}
Then, we have
$$\sum_{i=0}^{\infty}\omega(\kappa^{i}r)\leq C\int_{0}^{r}\frac{\omega(t)}{t}\ dt,$$
where $C=C(\kappa,c_{1},c_{2})$.
\end{lemma}

\cite[Lemma 4.1]{dek} can be extended as follows by replacing the exponent $p=2$ with a general $p\in(1,\infty)$. See also \cite[Lemma 6.3]{dx}.
\begin{lemma}\label{lemma weak II}
Let $\cD$ be a bounded domain in $\mathbb R^{n}$ satisfying
\begin{equation}\label{condition D}
|\cD_{r}(x)|\geq A_{0}r^{n}\quad\mbox{for~all}~x\in\overline{\cD}~\mbox{and}~r\in(0,\mbox{diam}\ \cD],
\end{equation}
where $A_{0}>0$ is a constant. Let $p\in (1,\infty)$ and $T$ be a bounded linear operator on $L^{p}(\cD)$. Suppose that for any $\bar{y}\in \cD$ and $0<r<\mu~ \mbox{diam}\ \cD$,  we have
$$\int_{\cD\setminus B_{cr}(\bar{y})}|Tb|\leq C_{0}\int_{\cD_{r}(\bar{y})}|b|$$
whenever $b\in L^{p}(\cD)$ is supported in $\cD_{r}(\bar{y})$, $\int_{\cD}b=0$, and $c>1$, $C_{0}>0$, $\mu\in(0,1)$ are constants. Then for any $g\in L^{p}(\cD)$ and any $t>0$, we have
$$|\{x\in \cD: |Tg(x)|>t\}|\leq\frac{C}{t}\int_{\cD}|g|,$$
where $C=C(n,c,C_{0},\cD,A_{0},\mu,\|T\|_{L^p\rightarrow L^p})$ is a constant.
\end{lemma}


\section{Proofs of Theorem \ref{thm} and Corollary \ref{coro}}\label{proof thm}

In this section, we shall first consider the solution $u\in W^{2,p}(\cD)$ to the equation \eqref{equations} without lower-order terms. Then for the general case, we move lower-order terms to the right-hand side and use the $L^{p}$-estimates in Lemma \ref{nondiv lemma}. 

\subsection{Proof of Theorem \ref{thm}}

We first assume that $b^{i}\equiv c\equiv0$. The general case will be outlined at the end of the proof. We fix $x_{0}\in B_{3/4}\cap \cD_{j_{0}}$, $0<r\leq 1/4$, and take a coordinate system associated with $x_{0}$ as in Section \ref{subsection domain}. We shall derive an a priori estimate of the modulus of continuity of $D_{xx'}u$ by assuming that $u\in C^{1,1}(B_{3/4})$. Denote
\begin{align*}
\bar{L}_{x'_{0}}u:=\bar{a}^{ij}(x'_{0},x^{n})D_{ij}u.
\end{align*}
As before, we modify the coefficients $\bar{a}^{ij}(x'_{0},x^{n})$ to get the following elliptic operator defined by
$$
\tilde{L}_{x'_{0}}u:=\tilde{a}^{ij}D_{ij}u,
$$
where $\tilde{a}^{ij}=\eta \bar{a}^{ij}(x'_{0},x^{n})+\delta(1-\eta)\delta_{ij}$  with $\eta\in C_{0}^{\infty}(B_{r}(x_{0}))$ satisfying
$$0\leq\eta\leq1,\quad\eta\equiv1~\mbox{in}~B_{2r/3}(x_{0}),
\quad|\nabla\eta|\leq {5}/{r}.$$
We present several lemmas (and their proofs) that will provide key estimates for the proof of Theorem \ref{thm}.

\begin{lemma}\label{weak est v}
Let $p\in(1,\infty)$ and $v\in W^{2,p}(B_{r}(x_{0}))\cap W_{0}^{1,p}(B_{r}(x_{0}))$ be a unique solution of
\begin{align*}
\begin{cases}
\tilde{L}_{x'_{0}}v=F\chi_{B_{r/2}(x_{0})}&~\mbox{in}~B_{r}(x_{0}),\\
v=0&~\mbox{on}~\partial B_{r}(x_{0}),
\end{cases}
\end{align*}
where $F\in L^p(B_{r/2}(x_0))$. Then for any $t>0$, we have
\begin{align*}
|\{x\in B_{r/2}(x_{0}): |D^{2}v(x)|>t\}|\leq\frac{C}{t}\|F\|_{L^{1}(B_{r/2}(x_{0}))},
\end{align*}
where $C=C(n,p,\delta)>0$.
\end{lemma}

\begin{proof}
For simplicity, we set $x_{0}=0$ and $r=1$. We modify the proof of \cite[Lemma 2.12]{dlk}. By using Lemma \ref{solvability}, we can see that the map $T: F\mapsto D^{2}v$ is a bounded linear operator on $L^{p}(B_{1/2})$, it suffices to show  that $T$ satisfies the hypothesis of Lemma \ref{lemma weak II}. We introduce a new matrix-valued function $\hat{a}^{ij}={\tilde{a}^{ij}}/{\tilde{a}^{nn}}$, so that $\hat{a}^{nn}=1$. Clearly, $\hat{a}^{ij}$ satisfies the ellipticity and boundedness conditions with a new ellipticity constant determined by $\delta$. Therefore, $v\in W^{2,p}(B_{1})\cap W_{0}^{1,p}(B_{1})$ satisfies
\begin{align*}
\begin{cases}
\hat{a}^{ij}D_{ij}v=\hat{F}\chi_{B_{1/2}}&\quad\mbox{in}~B_{1},\\
v=0&\quad\mbox{on}~\partial B_{1},
\end{cases}
\end{align*}
where $\hat{F}=F/{\tilde{a}^{nn}}$. It is sufficient to show
\begin{align*}
|\{x\in B_{1/2}: |D^{2}v(x)|>t\}|\leq\frac{C}{t}\|\hat{F}\|_{L^{1}(B_{1/2})}.
\end{align*}
We take $c=24$ and fix $\bar{y}\in B_{1/2}$, $r\in (0,1/4)$. Let $b\in L^{p}(B_{1})$ be supported in $B_{r}(\bar{y})\cap B_{1/2}$ with mean zero and $v_{1}\in W^{2,p}(B_{1})\cap W_{0}^{1,p}(B_{1})$ be a solution of
\begin{align}\label{eq v1}
\begin{cases}
\hat{a}^{ij}D_{ij}v_{1}=b&\quad\mbox{in}~B_{1},\\
v_{1}=0&\quad\mbox{on}~\partial B_{1},
\end{cases}
\end{align}
the solvability of which follows from Lemma \ref{solvability}.

For any $R\in [cr,1]$ such that $B_{1/2}\setminus B_{R}(\bar{y})\neq\emptyset$ and $g=(g^{ij})_{i,j=1}^{n}\in C_{0}^{\infty}((B_{2R}(\bar{y})\setminus B_{R}(\bar{y}))\cap B_{1/2})$, let $v_{2}\in W_{0}^{1,p'}(B_{1})$ be a weak solution of
\begin{align*}
\begin{cases}
D_{i}(\breve{a}^{ij}D_{j}v_{2}+D_{j}\breve{a}^{ij}v_{2})=\Div^{2}g&\quad\mbox{in}~B_{1},\\
v_{2}=0&\quad\mbox{on}~\partial B_{1},
\end{cases}
\end{align*}
where ${1}/{p}+{1}/{p'}=1$ and $\breve{A}=(\breve{a}^{ij})$ are defined as follows,
\begin{equation}\label{tilde a}
\begin{split}
\breve{a}^{nn}=1;\quad \breve{a}^{ij}=\hat{a}^{ij}~ \mbox{for}~i, j\in\{1,\ldots,n-1\};\\
\breve{a}^{nj}=2\hat{a}^{nj}~ \mbox{and}~\breve{a}^{j n}=0~ \mbox{for}~ j\in\{1,\ldots,n-1\}.
\end{split}
\end{equation}
It is easy to check that $\breve{A}$ satisfies the ellipticity and boundedness conditions with a new ellipticity constant determined by $\delta$. Also, $D_{j}\breve{a}^{ij},\ j=1,\dots,n-1$ is bounded
and $D_{n}\breve{a}^{in}=0$. Since $g=0$, $\breve{a}^{ij}=\breve{a}^{ij}(x'_{0},x^{n})=:\breve{a}^{ij}(x^{n})$, and $\breve{a}^{nn}=1$ in $B_{R/12}(\bar{y})\subset B_{2/3}$, we get
$$D_{i}(\breve{a}^{ij}(x^{n})D_{j}v_{2})=0\quad\mbox{in}~B_{R/12}(\bar{y}).$$
By the De Giorgi-Nash-Moser estimate, we see that $v_{2}$ is H\"{o}lder continuous in $B_{r}(\bar{y})$ and
\begin{align}\label{De Giorgi}
[v_{2}]_{\gamma;B_{r}(\bar{y})}\leq [v_{2}]_{\gamma;B_{R/24}(\bar{y})}\leq CR^{-\gamma-\frac{n}{p'}}\|v_{2}\|_{L^{p'}(B_{R/12}(\bar{y}))},
\end{align}
where $\gamma\in(0,1)$ and $C>0$ depending only on $n,\delta$ and $\Lambda$. On the other hand, one observe that
\begin{align}\label{trans}
&\sum_{i,j=1}^{n}D_{i}(\breve{a}^{ij}D_{j}v_{2}+D_{j}\breve{a}^{ij}v_{2})\nonumber\\
&=\sum_{i,j=1}^{n-1}D_{i}(\hat{a}^{ij}D_{j}v_{2}+D_{j}\hat{a}^{ij}v_{2})+2\sum_{j=1}^{n-1}D_{n}(\hat{a}^{nj}D_{j}v_{2}+D_{j}\hat{a}^{nj}v_{2})
+D_{n}(\hat{a}^{nn}D_{n}v_{2})\nonumber\\
&=\sum_{i,j=1}^{n-1}D_{ij}(\hat{a}^{ij}v_{2})+2\sum_{j=1}^{n-1}D_{nj}(\hat{a}^{nj}v_{2})
+D_{nn}v_{2}=\sum_{i,j=1}^{n}D_{ij}(\hat{a}^{ij}v_{2}).
\end{align}
Here, we used the fact that $\hat{a}^{nn}=1$. Therefore, we see that $v_{2}$ is also an adjoint solution of
\begin{align}\label{eq v2}
\begin{cases}
D_{ij}(\hat{a}^{ij}v_{2})=\Div^{2}g&\quad\mbox{in}~B_{1},\\
v_{2}=0&\quad\mbox{on}~\partial B_{1}.
\end{cases}
\end{align}
Hence, by using Lemma \ref{sol adjoint}, we get
\begin{align}\label{es v2}
\|v_{2}\|_{L^{p'}(B_{1})}\leq C\|g\|_{L^{p'}(B_{1})}=
C\|g\|_{L^{p'}((B_{2R}(\bar{y})\setminus B_{R}(\bar{y}))\cap B_{1/2})}.
\end{align}
By \eqref{eq v1}, \eqref{eq v2}, and the hypothesis on $b$, we have
$$\int_{(B_{2R}(\bar{y})\setminus B_{R}(\bar{y}))\cap B_{1/2}}D_{ij}v_{1}g^{ij}=\int_{B_{r}(\bar{y})\cap B_{1/2}}bv_{2}=\int_{B_{r}(\bar{y})\cap B_{1/2}}b\big(v_{2}-v_2(\bar y)\big).$$
Then by using \eqref{De Giorgi} and \eqref{es v2}, we bound the absolute value of the right-hand side above by
\begin{align*}
&\|b\|_{L^{1}(B_{r}(\bar{y})\cap B_{1/2})}\|v_{2}-v_2(\bar y)\|_{L^{\infty}(B_{r}(\bar{y})\cap B_{1/2})}\leq\|b\|_{L^{1}(B_{r}(\bar{y})\cap B_{1/2})}[v_{2}]_{\gamma;B_{r}(\bar{y})\cap B_{1/2}}r^{\gamma}\\
&\leq C\big(\frac{r}{R}\big)^{\gamma}R^{-\frac{n}{p'}}\|b\|_{L^{1}(B_{r}(\bar{y})\cap B_{1/2})}\|v_{2}\|_{L^{p'}(B_{R/12}(\bar{y}))}\\
&\leq C\big(\frac{r}{R}\big)^{\gamma}R^{-\frac{n}{p'}}\|b\|_{L^{1}(B_{r}(\bar{y})\cap B_{1/2})}\|g\|_{L^{p'}((B_{2R}(\bar{y})\setminus B_{R}(\bar{y}))\cap B_{1/2})}.
\end{align*}
By duality, we have
$$\|D^{2}v_{1}\|_{L^{p}((B_{2R}(\bar{y})\setminus B_{R}(\bar{y}))\cap B_{1/2})}\leq C\big(\frac{r}{R}\big)^{\gamma}R^{-\frac{n}{p'}}\|b\|_{L^{1}(B_{r}(\bar{y})\cap B_{1/2})}.$$
Hence, by H\"{o}lder's inequality, we get
\begin{align}\label{dilation DDv}
\|D^{2}v_{1}\|_{L^{1}((B_{2R}(\bar{y})\setminus B_{R}(\bar{y}))\cap B_{1/2})}\leq C\big(\frac{r}{R}\big)^{\gamma}\|b\|_{L^{1}(B_{r}(\bar{y})\cap B_{1/2})}.
\end{align}
Let $N$ be the smallest positive integer such that $B_{1/2}\subset B_{2^{N}cr}(\bar{y})$. By taking $R=cr, 2cr,\ldots,2^{N-1}cr$ in \eqref{dilation DDv} and summarizing, we have
\begin{align*}
\int_{B_{1/2}\setminus B_{cr}(\bar{y})}|D^{2}v_{1}|\ dx\leq C\sum_{k=1}^{N}2^{-\gamma k}\|b\|_{L^{1}(B_{r}(\bar{y})\cap B_{1/2})}\leq C\int_{B_{r}(\bar{y})\cap B_{1/2}}|b|\ dx.
\end{align*}
Therefore, $T$ satisfies the hypothesis of Lemma \ref{lemma weak II}, and the proof is finished.
\end{proof}

Denote
$$\phi(x_{0},r):=\inf_{\mathbf q\in\mathbb R^{n\times(n-1)}}\left(\fint_{B_{r}(x_{0})}|D_{xx'}u-\mathbf q|^{q}\ dx\right)^{1/q},$$
where $q\in (0,1)$ is some fixed exponent. First of all, by H\"{o}lder's inequality, we have
\begin{align}\label{est phi2}
\phi(x_{0},r)\leq\left(\fint_{B_{r}(x_{0})}|D_{xx'}u|^{q}\ dx\right)^{1/q}\leq Cr^{-n}\|D_{xx'}u\|_{L^{1}(B_{r}(x_{0}))},
\end{align}
where $C=C(n)$.

\begin{lemma}\label{lemma itera}
For any $\gamma\in (0,1)$ and $0<\rho\leq r\leq 1/4$, we have
\begin{align}\label{est phi'}
\phi(x_{0},\rho)\leq C\Big(\frac{\rho}{r}\Big)^{\gamma}r^{-n}\|D_{xx'}u\|_{L^{1}(B_{r}(x_{0}))}+C\tilde{\omega}_{A}(\rho)\|D^{2}u\|_{L^{\infty}(B_{r}(x_{0}))}+C\tilde{\omega}_{f}(\rho),
\end{align}
where $C=C(n,p,\delta,\gamma)$, and $\tilde\omega_{\bullet}(t)$ is a Dini function derived from $\omega_{\bullet}(t)$.
\end{lemma}

\begin{proof}
For any $t>0$, by using Lemma \ref{weak est v} with $F=f(x)-\bar{f}(x'_{0},x^{n})+(\tilde{a}^{ij}(x)-a^{ij}(x))D_{ij}u$ and \eqref{est A}, we have
\begin{align}\label{weak type DDv}
&|\{x\in B_{r/2}(x_{0}): |D^{2}v(x)|>t\}|\leq\frac{C}{t}\int_{B_{r/2}(x_{0})}|F|\ dx\nonumber\\
&\leq\frac{C}{t}\left(\int_{B_{r/2}(x_{0})}|f(x)-\bar{f}(x'_{0},x^{n})|\ dx+\int_{B_{r/2}(x_{0})}|(\tilde{a}^{ij}(x)-a^{ij}(x))D_{ij}u|\ dx\right)\nonumber\\
&\leq\frac{C}{t}\Big(r^{n}\bar\omega_{f}(r)+r^{n}\bar\omega_{A}(r)\|D^{2}u\|_{L^{\infty}(B_{r}(x_{0}))}\Big),
\end{align}
where $\bar\omega_{\bullet}(r):=\omega_{\bullet}(r)+\omega_{1}(r)$. Therefore, for any given $q\in (0,1)$, we have
\begin{align*}
&\int_{B_{r/2}(x_{0})}|D^{2}v|^{q}\ dx=\int_{0}^{\infty}qt^{q-1}|\{x\in B_{r/2}(x_{0}): |D^{2}v(x)|>t\}|\ dt\\
&=\left(\int_{0}^{\tau}+\int_{\tau}^{\infty}\right)qt^{q-1}|\{x\in B_{r/2}(x_{0}): |D^{2}v(x)|>t\}|\ dt\\
&\leq C\tau^{q}|B_{r}(x_{0})|+\frac{Cq}{1-q}\tau^{q-1}\Big(r^{n}\bar\omega_{f}(r)+r^{n}\bar\omega_{A}(r)\|D^{2}u\|_{L^{\infty}(B_{r}(x_{0}))}\Big).
\end{align*}
By choosing a suitable $\tau$, we have
\begin{align}\label{holder v}
\left(\fint_{B_{r/2}(x_{0})}|D^{2}v|^{q}\ dx\right)^{1/q}\leq C\Big(\bar\omega_{A}(r)\|D^{2}u\|_{L^{\infty}(B_{r}(x_{0}))}+\bar\omega_{f}(r)\Big).
\end{align}

Let $w=u-v$, which satisfies $\bar{L}_{x'_{0}}w=\bar{f}(x'_{0},x^{n})$ in $B_{r/2}(x_{0})$. By Lemma \ref{lemma DDDx'} with a suitable scaling, we see that for any $\mathbf q\in\mathbb R^{n\times(n-1)}$,
$$\|D^{2}D_{x'}w\|_{L^{\infty}(B_{r/4}(x_{0}))}^{q}
\leq Cr^{-(n+q)}\int_{B_{r/2}(x_{0})}|D_{xx'}w-\mathbf q|^{q}\ dx.$$
Hence, for any $\kappa\in (0, 1/4)$, we have
\begin{align*}
\|D_{xx'}w-(D_{xx'}w)_{B_{\kappa r}(x_{0})}\|_{L^{q}(B_{\kappa r}(x_{0}))}^{q}
&\leq C(\kappa r)^{n+q}\|D^{2}D_{x'}w\|_{L^{\infty}(B_{r/4}(x_{0}))}^{q}\\
&\leq C\kappa^{n+q}\int_{B_{r/2}(x_{0})}|D_{xx'}w-\mathbf q|^{q}\ dx.
\end{align*}
That is,
\begin{align}\label{ineq DDw}
\left(\fint_{B_{\kappa r}(x_{0})}|D_{xx'}w-(D_{xx'}w)_{B_{\kappa r}(x_{0})}|^{q}\ dx\right)^{1/q}
\leq C_{0}\kappa\left(\fint_{B_{r/2}(x_{0})}|D_{xx'}w-\mathbf{q}|^{q}\ dx\right)^{1/q},
\end{align}
where $C_{0}>0$ is a constant depending on $n,p,\delta$, and $\Lambda$. Recalling that $u=w+v$, by using \eqref{holder v} and \eqref{ineq DDw}, we obtain
\begin{align*}
&\left(\fint_{B_{\kappa r}(x_{0})}|D_{xx'}u-(D_{xx'}w)_{B_{\kappa r}(x_{0})}|^{q}\ dx\right)^{1/q}\\
&\leq C_{0}\kappa\left(\fint_{B_{r/2}(x_{0})}|D_{xx'}u-\mathbf{q}|^{q}\ dx\right)^{1/q}
+C\kappa^{-\frac{n}{q}}\Big(\bar\omega_{A}(r)\|D^{2}u\|_{L^{\infty}(B_{r}(x_{0}))}+\bar\omega_{f}(r)\Big).
\end{align*}
Since $\mathbf{q}\in\mathbb R^{n\times(n-1)}$ is arbitrary, we obtain
\begin{align*}
\phi(x_{0},\kappa r)\leq C_{0}\kappa\phi(x_{0},r)+C\kappa^{-\frac{n}{q}}\Big(\bar\omega_{A}(r)\|D^{2}u\|_{L^{\infty}(B_{r}(x_{0}))}+\bar\omega_{f}(r)\Big).
\end{align*}
For any $\gamma\in(0,1)$, fix a $\kappa\in(0, 1/4)$ sufficiently small such that $C_{0}\kappa\leq\kappa^{\gamma}$. Then
\begin{align*}
\phi(x_{0},\kappa r)\leq \kappa^{\gamma}\phi(x_{0},r)+C\Big(\bar\omega_{A}(r)\|D^{2}u\|_{L^{\infty}(B_{r}(x_{0}))}+\bar\omega_{f}(r)\Big).
\end{align*}
By iterating, for $j=1,2,\ldots$, we obtain
\begin{align}\label{phi DDu}
\phi(x_{0},\kappa^{j}r)\leq\kappa^{j\gamma}\phi(x_{0},r)+C\tilde\omega_{A}\|D^{2}u\|_{L^{\infty}(B_{r}(x_{0}))}
(\kappa^{j}r)+C\tilde\omega_{f}(\kappa^{j}r),
\end{align}
where
\begin{align}\label{tilde phi}
\tilde\omega_{\bullet}(t)=\sum_{i=1}^{\infty}\kappa^{i\gamma}\Big(\bar\omega_{\bullet}(\kappa^{-i}t)\chi_{\kappa^{-i}t\leq1}+\bar\omega_{\bullet}(1)\chi_{\kappa^{-i}t>1}\Big),
\end{align}
which is a Dini function; see \cite[Lemma 1]{d}, and satisfies \eqref{equivalence}.

Now, for any $\rho$ satisfying $0<\rho\leq r\leq 1/4$, we take $j$ to be the integer satisfying $\kappa^{j+1}<\rho/r\leq\kappa^{j}$. Then, by \eqref{phi DDu} and \eqref{equivalence}, we have
\begin{align}\label{itera phi}
\phi(x_{0},\rho)\leq C\Big(\frac{\rho}{r}\Big)^{\gamma}\phi(x_{0},r)+C\tilde{\omega}_{A}(\rho)\|D^{2}u\|_{L^{\infty}(B_{r}(x_{0}))}+C\tilde{\omega}_{f}(\rho).
\end{align}
Hence, \eqref{est phi'} follows from \eqref{est phi2} and \eqref{itera phi}.
\end{proof}

\begin{lemma}
We have
\begin{align}\label{est DDu''}
\|D^{2}u\|_{L^{\infty}(B_{1/4})}\leq C\|D^2u\|_{L^{1}(B_{3/4})}+C\left(\int_{0}^{1}\frac{\tilde\omega_{f}(t)}{t}\ dt+\|f\|_{L^{\infty}(B_{1})}\right),
\end{align}
where $C>0$ is a constant depending only on $n,p,\delta,\gamma$, and $\omega_{A}$.
\end{lemma}

\begin{proof}
Let $\kappa\in(0, 1/4)$ be the constant in the proof of Lemma \ref{lemma itera}, $\mathbf \{q_{x_{0},\kappa^{K} r}\}_{K=0}^{\infty}\in\mathbb R^{n\times(n-1)}$ be such that
$$\phi(x_{0},\kappa^{K} r)=\left(\fint_{B_{\kappa^{K} r}(x_{0})}|D_{xx'}u-\mathbf q_{x_{0},\kappa^{K}r}|^{q}\ dx\right)^{1/q}.$$
Since
$$|\mathbf q_{x_{0},\kappa r}-\mathbf q_{x_{0},r}|^{q}\leq|D_{xx'}u-\mathbf q_{x_{0},r}|^{q}+|D_{xx'}u-\mathbf q_{x_{0},\kappa r}|^{q},$$
by taking average over $x\in B_{\kappa r}(x_{0})$ and taking the $q$-th root, we obtain
$$|\mathbf q_{x_{0},\kappa r}-\mathbf q_{x_{0},r}|\leq C(\phi(x_{0},\kappa r)+\phi(x_{0},r)).$$
By iterating, we have
\begin{align}\label{mathbf q}
|\mathbf q_{x_{0},\kappa^{K}r}-\mathbf q_{x_{0},r}|\leq C\sum_{j=0}^{K}\phi(x_{0},\kappa^{j}r).
\end{align}
Notice that \eqref{phi DDu} implies
$$\lim_{K\rightarrow\infty}\phi(x_{0},\kappa^{K}r)=0.$$
Thus, by using the assumption that $Du\in C^{0,1}(B_{3/4})$ and the Lebesgue differentiation theorem, we obtain for {\em a.e.} $x_{0}\in B_{3/4}$,
$$\lim_{K\rightarrow\infty}\mathbf q_{x_{0},\kappa^{K}r}=D_{xx'}u(x_{0}).$$
On the other hand, \eqref{tilde phi} implies that $\tilde\omega_{A}$ and $\tilde\omega_{f}$ satisfy \eqref{equivalence}. Therefore, by taking $K\rightarrow\infty$ in \eqref{mathbf q}, using \eqref{phi DDu} and Lemma \ref{lemma omiga}, for {\em a.e.} $x_{0}\in B_{3/4}$, we have
\begin{align}\label{est DDu q}
&|D_{xx'}u(x_{0})-\mathbf q_{x_{0},r}|\leq C\sum_{j=0}^{\infty}\phi(x_{0},\kappa^{j}r)\nonumber\\
&\leq C\left(\phi(x_{0},r)+\|D^{2}u\|_{L^{\infty}(B_{r}(x_{0}))}\int_{0}^{r}\frac{\tilde\omega_{A}(t)}{t}\ dt+\int_{0}^{r}\frac{\tilde\omega_{f}(t)}{t}\ dt\right).
\end{align}
By averaging the inequality
$$|\mathbf q_{x_{0},r}|^{q}\leq|D_{xx'}u-\mathbf q_{x_{0},r}|^{q}+|D_{xx'}u|^{q}$$
over $x\in B_{r}(x_{0})$ and taking the $q$-th root, we have
$$|\mathbf q_{x_{0},r}|\leq2^{1/q-1}\phi(x_{0},r)+2^{1/q-1}\left(\fint_{B_{r}(x_{0})}|D_{xx'}u|^{q}\ dx\right)^{1/q}.$$
Therefore, combining \eqref{est DDu q} and \eqref{est phi2}, we obtain for {\em a.e.} $x_{0}\in B_{3/4}$,
\begin{align*}
|D_{xx'}u(x_{0})|&\leq Cr^{-n}\|D_{xx'}u\|_{L^{1}(B_{r}(x_{0}))}\\
&\quad+C\left(\|D^{2}u\|_{L^{\infty}(B_{r}(x_{0}))}\int_{0}^{r}\frac{\tilde\omega_{A}(t)}{t}dt+\int_{0}^{r}\frac{\tilde\omega_{f}(t)}{t}dt\right).
\end{align*}
For any $x_{1}\in B_{1/4}$ and $0<r< 1/4$, we take the supremum of the above inequality over $B_{r}(x_{1})$ to get
\begin{align*}
\|D_{xx'}u\|_{L^{\infty}(B_{r}(x_{1}))}&\leq Cr^{-n}\|D^2u\|_{L^{1}(B_{2r}(x_{1}))}\\
&\quad+C\left(\|D^{2}u\|_{L^{\infty}(B_{2r}(x_{1}))}\int_{0}^{r}\frac{\tilde\omega_{A}(t)}{t}\ dt+\int_{0}^{r}\frac{\tilde\omega_{f}(t)}{t}\ dt\right).
\end{align*}
Recalling that $a^{ij}(x)D_{ij}u(x)=f(x)$, one can see that
$$D_{nn}u=\frac{1}{a^{nn}}\left(f-\sum_{(i,j)\neq(n,n)}a^{ij}D_{ij}u\right).$$
Therefore, we have
\begin{align*}
\|D^{2}u\|_{L^{\infty}(B_{r}(x_{1}))}&\leq Cr^{-n}\|D^2u\|_{L^{1}(B_{2r}(x_{1}))}
+C\Big(\|D^{2}u\|_{L^{\infty}(B_{2r}(x_{1}))}\int_{0}^{r}\frac{\tilde\omega_{A}(t)}{t}\ dt\\
&\quad+\int_{0}^{r}\frac{\tilde\omega_{f}(t)}{t}\ dt+\|f\|_{L^{\infty}(B_{1})}\Big).
\end{align*}
We fix $r_{0}< 1/4$ such that for any $0<r\leq r_{0}$,
$$C\int_{0}^{r}\frac{\tilde\omega_{A}(t)}{t}dt\leq\frac{1}{4^{n}}.$$
Then, for any $x_{1}\in B_{1/4}$ and $0<r\leq r_{0}$, we get
\begin{align*}
\|D^{2}u\|_{L^{\infty}(B_{r}(x_{1}))}&\leq4^{-n}\|D^{2}u\|_{L^{\infty}(B_{2r}(x_{1}))}+Cr^{-n}\|D^2u\|_{L^{1}(B_{2r}(x_{1}))}\\
&\quad+C\left(\int_{0}^{r}\frac{\tilde\omega_{f}(t)}{t}\ dt+\|f\|_{L^{\infty}(B_{1})}\right).
\end{align*}
For $k=1,2,\ldots$, denote $r_{k}=3/4-(1/2)^{k}$. For $x_{1}\in B_{r_{k}}$ and $r=(1/2)^{k+2}$, we have $B_{2r}(x_{1})\subset B_{r_{k+1}}$. We take $k_{0}\geq1$ sufficiently large such that $(1/2)^{k_{0}+2}\leq r_{0}$. It follows that for any $k\geq k_{0}$,
\begin{align*}
\|D^{2}u\|_{L^{\infty}(B_{r_{k}})}&\leq4^{-n}\|D^{2}u\|_{L^{\infty}(B_{r_{k+1}})}+C2^{kn}\|D^2u\|_{L^{1}(B_{3/4})}+C\left(\int_{0}^{1}\frac{\tilde\omega_{f}(t)}{t}\ dt+\|f\|_{L^{\infty}(B_{1})}\right).
\end{align*}
By multiplying the above by $4^{-kn}$ and summing over $k=k_{0},k_{0}+1,\ldots$, we have
\begin{align*}
&\sum_{k=k_{0}}^{\infty}4^{-kn}\|D^{2}u\|_{L^{\infty}(B_{r_{k}})}\\
&\leq\sum_{k=k_{0}+1}^{\infty}4^{-(k+1)n}\|D^{2}u\|_{L^{\infty}(B_{r_{k+1}})}
+C\|D^2u\|_{L^{1}(B_{3/4})}+C\left(\int_{0}^{1}\frac{\tilde\omega_{f}(t)}{t}\ dt+\|f\|_{L^{\infty}(B_{1})}\right).
\end{align*}
Recalling the assumption that $u\in C^{1,1}(B_{3/4})$, the summations on both sides are convergent, and we finally obtain \eqref{est DDu''}. The lemma is proved.
\end{proof}

Now we are ready to prove Theorem \ref{thm}.

{\bf Proof of Theorem \ref{thm}.}
By \eqref{est DDu q}, for $r\in (0,1/8)$, we have
\begin{align}\label{sup Dxx'}
&\sup_{x_{0}\in B_{1/8}}|D_{xx'}u(x_{0})-\mathbf q_{x_{0},r}|\nonumber\\
&\leq C\sup_{x_{0}\in B_{1/8}}\phi(x_{0},r)+C\|D^{2}u\|_{L^{\infty}(B_{1/4})}\int_{0}^{r}\frac{\tilde\omega_{A}(t)}{t}\ dt+C\int_{0}^{r}\frac{\tilde\omega_{f}(t)}{t}\ dt=:C\psi(r).
\end{align}
We recall that for each $x_0$, the coordinate system and thus $x'$ are chosen according to $x_0$.
By Lemma \ref{lemma itera}, for any $r\in (0,1/8)$, we obtain
\begin{align}\label{sup phi}
\sup_{x_{0}\in B_{1/8}}\phi(x_{0},r)\leq C\left(r^{\gamma}\|D^2u\|_{L^{1}(B_{1/4})}+\tilde{\omega}_{A}(r)\|D^{2}u\|_{L^{\infty}(B_{1/4})}+\tilde{\omega}_{f}(r)\right).
\end{align}
Suppose that $y\in B_{1/8}\cap \cD_{j_{1}}, j_{1}\in[1,l+1]$. Clearly, if $|x_{0}-y|\geq {1}/{32}$, then
\begin{align}\label{C2 est2}
&|D_{xx'}u(x_{0})-D_{xx'}u(y)|\leq2\|D^{2}u\|_{L^{\infty}(B_{1/4})}\nonumber\\
&\leq
C|x_{0}-y|^{\gamma}\left( \|D^2u\|_{L^{1}(B_{3/4})}+\int_{0}^{1}\frac{\tilde\omega_{f}(t)}{t}\ dt+\|f\|_{L^{\infty}(B_{1})}\right),
\end{align}
where we used \eqref{est DDu''} in the second inequality. Otherwise, if $|x_{0}-y|<{1}/{32}$, we set $r=|x_{0}-y|$ and discuss it further according to the following two cases:

{\bf Case 1.} If $$r\leq 1/16\max\{\dist(x_{0},\partial \cD_{j_{0}}),\dist(y,\partial \cD_{j_{1}})\},$$
then $j_{0}=j_{1}$. We define
\begin{align*}
\varphi(x_{0},r):=\inf_{\mathbf Q\in\mathbb R^{n\times n}}\left(\fint_{B_{r}(x_{0})}|D_{x}^{2}u-\mathbf Q|^{q}\ dx\right)^{1/q}.
\end{align*}
For any $\mathbf q=(q_{ij})\in \mathbb R^{n\times(n-1)}$, we define $\widetilde{\mathbf Q}:=(\widetilde{Q}_{ij})\in\mathbb R^{n\times n}$
by
\begin{equation}\label{def Q}
\begin{split}
\widetilde{Q}_{ij}&=q_{ij}\quad\mbox{for}~i=1,\dots,n,~j=1,\dots,n-1;\quad \widetilde{Q}_{in}=q_{ni}\quad\mbox{for}~i=1,\dots,n-1;\\
\widetilde{Q}_{nn}&=\frac{1}{\hat{a}^{nn}}\Big(\hat{f}-
\sum_{(i,j)\neq(n,n)}\hat{a}^{ij}\widetilde{Q}_{ij}\Big),
\end{split}
\end{equation}
where $\hat{a}^{ij}$ and $\hat{f}$ are constant functions corresponding to $a^{ij}$ and $f$, respectively.
Combining
$$D_{nn}u(x)=\frac{1}{a^{nn}(x)}\Big(f-\sum_{(i,j)\neq(n,n)}a^{ij}D_{ij}u\Big),$$
\eqref{def Q} and \eqref{est phi'}, we reach the following estimate: for any $\gamma\in (0,1)$ and $0<\rho\leq r<1/8$, we have
\begin{align*}
\varphi(x_{0},\rho)&\leq C\big(\phi(x_{0},\rho)+\omega_{f}(\rho)+\omega_{A}(\rho)(\|f\|_{L^{\infty}(B_{\rho}(x_{0}))}+\|D_{xx'}u\|_{L^{\infty}(B_{\rho}(x_{0}))})\big).
\end{align*}
By using the same argument that led to \eqref{sup Dxx'}, we obtain
\begin{align}\label{sup Dx 2}
&\sup_{x_{0}\in B_{1/8}}|D_{x}^{2}u(x_{0})-\mathbf Q_{x_{0},r}|\nonumber\\
&\leq C\sup_{x_{0}\in B_{1/8}}\phi(x_{0},r)+C\big(\|D^{2}u\|_{L^{\infty}(B_{1/4})}+\|f\|_{L^{\infty}(B_{1/4})}\big)\int_{0}^{r}\frac{\tilde\omega_{A}(t)}{t}\ dt+C\int_{0}^{r}\frac{\tilde\omega_{f}(t)}{t}\ dt,
\end{align}
where $\mathbf Q_{x_{0},r}\in\mathbb R^{n\times n}$ satisfying
$$\varphi(x_{0},r)=\left(\fint_{B_{r}(x_{0})}|D_{x}^{2}u-\mathbf Q_{x_{0},r}|^{q}\ dx\right)^{1/q}.$$
Then by the triangle inequality, we have
\begin{align}\label{case1}
&|D_{x}^{2}u(x_{0})-D_{x}^{2}u(y)|^{q}\nonumber\\
&\leq |D_{x}^{2}u(x_{0})-\mathbf Q_{x_{0},r}|^{q}+|D_{x}^{2}u(z)-\mathbf Q_{x_{0},r}|^{q}+|D_{x}^{2}u(z)-X^{\top}\mathbf Q_{y,r}X|^{q}\nonumber\\
&\quad+|D_{x}^{2}u(y)-X^{\top}\mathbf Q_{y,r}X|^{q},
\end{align}
where $X=(X_{ij})$ is an $n\times n$ matrix, and $X_{ij}=\frac{\partial y^{i}}{\partial x^{j}}$ for $i,j=1,\dots,n$. We use $D_{y}$ to denote derivatives in the coordinate system associated with $y$, so that
$$D_{x}^{2}u(y)=X^{\top}D_{y}^{2}u(y)X.$$
Therefore, similar to \eqref{sup Dx 2}, we have
\begin{align*}
&|D_{x}^{2}u(y)-X^{\top}\mathbf Q_{y,r}X|=|X^{\top}(D_{y}^{2}u(y)-\mathbf Q_{y,r})X|\leq C|D_{y}^{2}u(y)-\mathbf Q_{y,r}|\\
&\leq C\phi(y,r)+C\big(\|D^{2}u\|_{L^{\infty}(B_{1/4})}+\|f\|_{L^{\infty}(B_{1/4})}\big)\int_{0}^{r}\frac{\tilde\omega_{A}(t)}{t}\ dt+C\int_{0}^{r}\frac{\tilde\omega_{f}(t)}{t}\ dt.
\end{align*}
We take the average over $z\in B_{r}(x_{0})\cap B_{r}(y)$ in \eqref{case1}, and then take the $q$-th root to get
$$|D_{x}^{2}u(x_{0})-D_{x}^{2}u(y)|\leq C\left(\psi(r)+\|f\|_{L^{\infty}(B_{1/4})}\int_{0}^{r}\frac{\tilde{\omega}_{A}(t)}{t}\ dt\right).$$
By using \eqref{est DDu''}, \eqref{sup Dxx'}, and \eqref{sup phi}, we obtain
\begin{align}\label{C2 est}
&|D_{x}^{2}u(x_{0})-D_{x}^{2}u(y)|\nonumber\\
&\leq C|x_{0}-y|^{\gamma}\|D^2u\|_{L^{1}(B_{3/4})}+C\int_{0}^{|x_{0}-y|}\frac{\tilde{\omega}_{f}(t)}{t}\ dt\nonumber\\
&\quad+C\int_{0}^{|x_{0}-y|}\frac{\tilde{\omega}_{A}(t)}{t}\ dt\left(\|D^2u\|_{L^{1}(B_{3/4})}+\int_{0}^{1}\frac{\tilde{\omega}_{f}(t)}{t}\ dt+\|f\|_{L^{\infty}(B_{1})}\right),
\end{align}
where $\gamma\in(0,1)$ is arbitrary.

{\bf Case 2.} If $r>1/16\max\{\dist(x_{0},\partial \cD_{j_{0}}),\dist(y,\partial \cD_{j_{1}})\}$, then by the triangle inequality, we have
\begin{align}\label{case2}
&|D_{xx'}u(x_{0})-D_{xx'}u(y)|^{q}\nonumber\\
&\leq|D_{xx'}u(x_{0})-\mathbf q_{x_{0},r}|^{q}+|\mathbf q_{x_{0},r}-\mathbf q_{y,r}|^{q}+|D_{yy'}u(y)-\mathbf q_{y,r}|^{q}+|D_{yy'}u(y)-D_{xx'}u(y)|^{q}\nonumber\\
&\leq C\psi^{q}(r)+|D_{xx'}u(z)-\mathbf q_{x_{0},r}|^{q}+|D_{yy'}u(z)-\mathbf q_{y,r}|^{q}+|D_{yy'}u(z)-D_{xx'}u(z)|^{q}\nonumber\\
&\quad+|D_{yy'}u(y)-D_{xx'}u(y)|^{q},\quad\forall~z\in B_{r}(x_{0})\cap B_{r}(y).
\end{align}
In order to estimate the last two terms in \eqref{case2}, we first notice that
\begin{align}\label{coordinate}
D_{yy'}u(y)-D_{xx'}u(y)&=X^{-\top}D_{x}^{2}u(y)X^{-1}I_{1}-D_{x}^{2}u(y)I_{1}\nonumber\\
&=\big(X^{-\top}-I\big)D_{x}^{2}u(y)I_{1}+X^{-\top}D_{x}^{2}u(y)\big(X^{-1}-I\big)I_{1},
\end{align}
where $I$ is the $n\times n$ identity matrix and $I_{1}=(\delta_{ij})$ is an $n\times (n-1)$ matrix.
On the other hand, we suppose that the closest point on $\partial \cD_{j_{1}}$ to $y$ is $(y',h_{j_{1}}(y'))$, and let $$
\nu_{2}=\frac{\big(-\nabla_{x'}h_{j_{1}}(y'),1\big)^{\top}}{\sqrt{1+|\nabla_{x'}h_{j_{1}}(y')|^{2}}}
$$
be the unit normal vector at $(y',h_{j_{1}}(y'))$ on the surface $\{(y',t): t=h_{j_{1}}(y')\}$. The corresponding tangential vectors are
\begin{align*}
\tau_{2,1}=(1,0,\ldots,0,D_{x^{1}}h_{j_{1}}(y'))^{\top},\ \dots,\
\tau_{2,n-1}=(0,0,\ldots,1,D_{x^{n-1}}h_{j_{1}}(y'))^{\top}.
\end{align*}
We define the projection operator by
$$\mbox{proj}_{a}b=\frac{\langle a,b\rangle}{\langle a,a\rangle}a,$$
where $\langle a,b\rangle$ denotes the inner product of the vectors $a$ and $b$.
Then apply the Gram-Schmidt process as follows:
\begin{align*}
\hat{\tau}_{2,1}&=\tau_{2,1},\quad\tilde{\tau}_{2,1}=\frac{\hat{\tau}_{2,1}}{|\hat{\tau}_{2,1}|},\\
\hat{\tau}_{2,2}&=\tau_{2,2}-\mbox{proj}_{\hat{\tau}_{2,1}}\tau_{2,2},\quad\tilde{\tau}_{2,2}
=\frac{\hat{\tau}_{2,2}}{|\hat{\tau}_{2,2}|},\\
&\vdots\\
\hat{\tau}_{2,n-1}&=\tau_{2,n-1}-\sum_{j=1}^{n-2}\mbox{proj}_{\hat{\tau}_{2,j}}\tau_{2,n-1},
\quad\tilde{\tau}_{2,n-1}=\frac{\hat{\tau}_{2,n-1}}{|\hat{\tau}_{2,n-1}|}.
\end{align*}
Similarly, we denote $\nu_{1}=
(0',1)^{\top}$ to be the unit normal vector at $(x'_{0},h_{j_{0}}(x'_{0}))$, and the corresponding tangential vectors are
\begin{align*}
\tau_{1,1}=(1,0,\ldots,0,0)^{\top},\ldots,
\tau_{1,n-1}=(0,0,\ldots,1,0)^{\top}.
\end{align*}
It follows from the proof of Lemma \ref{volume} that the upper bound of $|\nabla_{x'}h_{j}(y')|$ is $C\omega_{1}(r)$, $j=1,\dots,M$. Then we have
\begin{align*}
|\nu_{1}-\nu_{2}|=\left|(0',1)^{\top}-\frac{\big(-\nabla_{x'}h_{j_{1}}(y'),1\big)^{\top}}{\sqrt{1+|\nabla_{x'}h_{j_{1}}(y')|^{2}}}\right|
&\leq C\omega_{1}(|x_{0}-y|),
\end{align*}
which is also true for $|\tau_{1,i}-\tilde{\tau}_{2,i}|, i=1,\ldots,n-1$. Thus, coming back to \eqref{coordinate}, we obtain
\begin{align}\label{diff DDu}
|D_{xx'}u(y)-D_{yy'}u(y)|\leq C\|D^{2}u\|_{L^{\infty}(B_{1/4})}\omega_{1}(|x_{0}-y|).
\end{align}
The penultimate term of \eqref{case2} is also bounded by the right-hand side of \eqref{diff DDu}. Coming back to \eqref{case2}, we take the average over $z\in B_{r}(x_{0})\cap B_{r}(y)$ and take the $q$-th root to get
\begin{align*}
&|D_{xx'}u(x_{0})-D_{xx'}u(y)|\nonumber\\
&\leq C\Big(\psi(r)+\phi(x_{0},r)+\phi(y,r)+\|D^{2}u\|_{L^{\infty}(B_{1/4})}\omega_{1}(|x_{0}-y|)\Big)\nonumber\\
&\leq C\Big(\psi(r)+\|D^{2}u\|_{L^{\infty}(B_{1/4})}\omega_{1}(|x_{0}-y|)\Big).
\end{align*}
It follows from \eqref{est DDu''}, \eqref{sup Dxx'}, and \eqref{sup phi} that
\begin{align}\label{C2 est1}
&|D_{xx'}u(x_{0})-D_{xx'}u(y)|\nonumber\\
&\leq C|x_{0}-y|^{\gamma}\|D^2u\|_{L^{1}(B_{3/4})}+C\int_{0}^{|x_{0}-y|}\frac{\tilde{\omega}_{f}(t)}{t}\ dt\nonumber\\
&\quad
+C\int_{0}^{|x_{0}-y|}\frac{\tilde{\omega}_{A}(t)}{t}\ dt\cdot\left(\|D^2 u\|_{L^{1}(B_{3/4})}+\int_{0}^{1}\frac{\tilde{\omega}_{f}(t)}{t}\ dt+\|f\|_{L^{\infty}(B_{1})}\right).
\end{align}
Thus, we finish the proof of Theorem \ref{thm} without lower-order terms under the assumption that $u\in C^{1,1}(B_{3/4})$.

Now we remove the assumption that $u\in C^{1,1}(B_{3/4})$. For this, it follows from the interior regularity in \cite{dek} for the non-divergence form elliptic equations that we only need to show that for any $x_{0}\in\partial\cD_{m}, m=1,\ldots,M-1$, there is a neighborhood of $x_0$ in which $Du$ is Lipschitz. In the case when $\partial\cD_{m}$ is smooth, say $C^{2,\alpha}$ with $\alpha\in(0,1)$, we can use the technique of locally flattening the boundaries, and an approximation argument, which is similar to that in the proof of \cite[Theorem 1.1]{dx}. To be specific, from the assumption that $x_0$ belongs to the boundaries of at most two of the subdomains, we can find a small $r_0>0$ and a $C^{2,\alpha}$ diffeomorphism of flattening the boundary $\partial\cD_{m}\cap B_{r_0}(x_{0})$:
$y=\Phi(x)=(\Phi^{1}(x),\dots,\Phi^{n}(x))$,
which satisfies $\Phi(x_0)=0$, $\det D\Phi=1$, and
$$
\Phi(\partial\cD_{m}\cap B_{r_0}(x_{0}))=\Phi(B_{r_0}(x_{0}))\cap \{y^n=0\}.
$$
Let $\hat{u}(y):=u(x)$, which satisfies
$$\hat{a}^{ij}D_{ij}\hat{u}=\hat{h},$$
where $\hat{a}^{ij}(y)=D_{k}\Phi^{i}D_{l}\Phi^{j}a^{kl}(x)$, $\hat h(y)=\hat f(y)-a^{kl}(x)D_{kl}\Phi^{i}D_{i}\hat{u}$, and $\hat{f}(y)=f(x)$, which are also of piecewise Dini mean oscillation in $\Phi(B_{r_0}(x_0))$. Then, it suffices to show that $D\hat u$ is Lipschitz near $0$. We take the standard mollification of the coefficients $\hat{a}^{ij}$ and data $\hat{h}$ in the $y'$ direction with a parameter $\varepsilon>0$. Then we get a uniform Lipschitz estimate independent of $\varepsilon$ by using  \cite[Theorem 3]{d} and the a priori $W^{2,\infty}$ estimate in Lemma \ref{est DDu''}. Finally, we take the limit as $\varepsilon \searrow 0$ by following the proof of \cite[Theorem 3]{d}.

For $\cD_{m}$ with $C^{1,\text{Dini}}$ boundary, we shall approximate $\cD_{m}$ by a sequence of increasing smooth domains $\{\cD_{m}^{k}\}_{k=1}^{\infty}$, which can be constructed via the regularized distance function $\rho(x)$ such that $\rho(x)\sim \mbox{dist}(x,\partial\cD_{m})$ for any $x\in \cD_{m}$ close to $\partial\cD_{m}$, is of class $C_{loc}^{\infty}(\cD_{m})$,
and has uniform $C^{1,\text{Dini}}$-characteristics.
For the existence and properties of $\rho$, we refer the reader to \cite[Theorem 2.1]{l} and \cite[Lemma 5.1]{dlk}. We set
$$
\cD_{m}^{k}:=\big\{x\in \cD_{m}: \rho(x)>1/k\big\},\quad k=1,2,\dots,$$
which have uniform $C^{1,\text{Dini}}$-characteristics.
Note that $\partial\cD_{m}\cap B_{2r}(x_{0})$ can be represented by
$$x^{n}=\varphi(x'),$$
where $\varphi$ is a $C^{1,\text{Dini}}$ function. Similarly, for any fixed $k=1,2,\dots,$ $\partial\cD_{m}^{k}\cap B_{2r}(x_{0})$ can be represented by
$$x^{n}=\varphi_{k}(x'),$$
where $\varphi_{k},k=1,2\dots,$ are $C^{1,\text{Dini}}$ functions with uniform $C^{1,\text{Dini}}$-characteristics. Next we approximate $a^{ij}$ and $f$ by
$$
a_{k}^{ij}(x',x^{n})=a^{ij}(x',x^{n}+\varphi(x')-\varphi_{k}(x')),\quad f_{k}(x',x^{n})=f(x',x^{n}+\varphi(x')-\varphi_{k}(x')),
$$
which are of piecewise Dini mean oscillation in subdomains $B_{2r_0}(x_0)\cap \cD_{m}^{k}$ and $B_{2r_0}(x_0)\setminus \cD_{m}^{k}$ with uniform moduli of continuity, and as $k\rightarrow\infty$,
$$
a_{k}^{ij}\rightarrow a^{ij}\quad \text{\em a.e.},\quad f_{k}\rightarrow f\quad \text{in}\ L^p(B_{r_0}(x_0)).
$$
After that, we can find a sequence of solutions $u_{k}\in W^{2,p}(B_{r/2}(x_{0}))$ that converges to $u$ almost everywhere with a unform $C^{1,1}$ norm in $B_{r_0}(x_0)$, by modifying the coefficients $a_{k}^{ij}$ which is similar to the argument in Section \ref{sub auxiliary}. Finally, passing to the limit finishes the proof of Theorem \ref{thm} under the assumption that $b^{i}\equiv c\equiv0$.

For the general case, we rewrite \eqref{equations} as
$$a^{ij}D_{ij}u=f-b^{i}D_{i}u-cu=:f_{0}.$$
Then we have
\begin{align}\label{f0}
\omega_{f_{0}}(r)&\leq\omega_{f}(r)+\|Du\|_{L^{\infty}(B_{r}(x_{0}))}\omega_{b}(r)+r^{\gamma}[Du]_{\gamma;B_{r}(x_{0})}\|b\|_{L^{\infty}(B_{r}(x_{0}))}\nonumber\\
&\quad+\|u\|_{L^{\infty}(B_{r}(x_{0}))}\omega_{c}(r)+r^{\gamma}[u]_{\gamma;B_{r}(x_{0})}\|c\|_{L^{\infty}(B_{r}(x_{0}))},
\end{align}
where $\gamma\in(0,1)$ and $B_{r}(x_{0})\subset B_{3/4}$. By using Lemma \ref{nondiv lemma},
\begin{align}\label{general est}
&\|u\|_{L^{\infty}(B_{r}(x_{0}))}+\|Du\|_{L^{\infty}(B_{r}(x_{0}))}+[u]_{\gamma;B_{r}(x_{0})}+[Du]_{\gamma;B_{r}(x_{0})}\nonumber\\
&\leq C\big(\|u\|_{L^{1}(B_1)}+\|f\|_{L^{\infty}(B_1)}\big).
\end{align}
Therefore, applying \eqref{general est} to \eqref{f0}, and using \eqref{C2 est1} (or \eqref{C2 est}, \eqref{C2 est2}), we have
\begin{align*}
&|D_{xx'}u(x_{0})-D_{xx'}u(y)|\\
&\leq C\int_{0}^{|x_{0}-y|}\frac{\tilde{\omega}_{A}(t)}{t}dt\cdot
\Bigg(\|D^2 u\|_{L^{1}(B_{3/4})}+\int_{0}^{1}\frac{\tilde{\omega}_{f_{0}}(t)}{t}dt+\|f_{0}\|_{L^{\infty}(B_{3/4})}\Bigg)\\
&\quad+C|x_{0}-y|^{\gamma}\|D^2 u\|_{L^{1}(B_{3/4})}+C\int_{0}^{|x_{0}-y|}\frac{\tilde{\omega}_{f_{0}}(t)}{t}dt\\
&\leq C\int_{0}^{|x_{0}-y|}\frac{\tilde{\omega}_{A}(t)}{t}dt\cdot\Bigg(\|D^2 u\|_{L^{1}(B_{3/4})}+\int_{0}^{1}\frac{\tilde{\omega}_{f}(t)}{t}dt
+\|f\|_{L^{\infty}(B_{1})}+\|u\|_{L^{1}(B_{1})}\Bigg)\\
&\quad+C|x_{0}-y|^{\gamma}\left(\|D^2 u\|_{L^{1}(B_{3/4})}+\|f\|_{L^{\infty}(B_{1})}+\|u\|_{L^{1}(B_{1})}\right)+C\int_{0}^{|x_{0}-y|}\frac{\tilde{\omega}_{f}(t)}{t}dt.
\end{align*}
Theorem \ref{thm} is proved.

\subsection{Proof of Corollary \ref{coro}}

Similar to the proof of Theorem \ref{thm}, we take $x_{0}\in B_{3/4}\cap \cD_{j_{0}}$. Let $A^{(j)}\in C^{\alpha}(\overline{\cD}_{j})$, $1\leq j\leq l+1$, be matrix-valued functions, $b^{(j)}$ and $f^{(j)}$ be in $C^{\alpha}(\overline{\cD}_{j})$. Define the piecewise constant (matrix-valued) functions
\begin{align*}
\bar{A}(x)=
A^{(j)}(x'_{0},h_{j}(x'_{0})),\quad x\in\Omega_{j}.
\end{align*}
From $b^{(j)}$ and $f^{(j)}$, we similarly define piecewise constant functions $\bar{b}$ and $\bar{f}$. By Lemma \ref{volume}, in this case we have $\omega_{1}(r)\sim r^{{\mu}/(1+\mu)}$. Therefore, we get the following result, which is similar to \cite[Lemma 5.2]{lv}.
\begin{lemma}\label{difference a holder}
Let $A, \bar{A}, b, \bar{b}, f$, and $\bar{f}$ be defined as above. There exists a constant $C>0$, depending only on $\max_{1\leq j\leq l+1}\|A\|_{C^{\alpha}(\overline{\cD}_{j})}$, $\max_{1\leq j\leq l+1}\|b\|_{C^{\alpha}(\overline{\cD}_{j})}$, $\max_{1\leq j\leq l+1}\|f\|_{C^{\alpha}(\overline{\cD}_{j})}$, $\max_{1\leq j\leq l+1}\|h_{j}\|_{C^{1,\mu}(\overline{\cD}_{j})}$ and $n,l,\mu,\alpha,\delta,\Lambda$, such that for any $x_0\in B_{3/4}$ and $r\in (0,1]$,
\begin{align*}
\fint_{B_{r}(x_{0})}|A-\bar{A}|\ dx+\fint_{B_{r}(x_{0})}|b-\bar{b}|\ dx
+\fint_{B_{r}(x_{0})}|f-\bar{f}|\ dx\leq Cr^{\alpha}.
\end{align*}
\end{lemma}
Corollary \ref{coro} directly follows from \eqref{C2 est2}, \eqref{C2 est}, and \eqref{C2 est1} by taking $\gamma\in(\alpha,1)$.

\section{Proofs of Theorem \ref{thm C0} and Corollary \ref{coro u}}\label{sec thm C0}

The proof of Theorem \ref{thm C0} is similar to that of Theorem \ref{thm}. Again, we first assume that $b^{i}\equiv c\equiv0$. The adjoint operator corresponding to $\bar{L}_{x'_{0}}$ is defined by
$$\bar{L}_{x'_{0}}^{*}u:=D_{ij}(\bar{a}^{ij}(x'_{0},x^{n})u).$$
Similarly, we define the modified operator
$$\tilde{L}_{x'_{0}}^{*}u:=D_{ij}(\tilde{a}^{ij}u).$$
Then, we have
$$
\tilde{L}_{x'_{0}}^{*}u=D_{ij}\Big(\big(\tilde{a}^{ij}(x)-a^{ij}(x)\big)u\Big)+\Div^{2} g.
$$
To prove Theorem \ref{thm C0}, we first present a lemma that is an adjoint version of Lemma \ref{weak est v}.

\begin{lemma}\label{lemma adjoint}
Let $p\in(1,\infty)$ and $v\in L^{p}(B_{r}(x_{0}))$ be a unique solution to the adjoint problem
\begin{align*}
\begin{cases}
\tilde{L}_{x'_{0}}^{*}v=\Div^{2}(G\chi_{B_{r/2}(x_{0})})&\quad\mbox{in}~B_{r}(x_{0}),\\
v=0&\quad\mbox{on}~\partial B_{r}(x_{0}),
\end{cases}
\end{align*}
where $G\in L^p(B_{r/2}(x_0))$. Then for any $t>0$, we have
$$|\{x\in B_{r/2}(x_{0}): |v(x)|>t\}|\leq\frac{C}{t}\|G\|_{L^{1}(B_{r/2}(x_{0}))},$$
where $C=C(n,p,\delta)>0$.
\end{lemma}

\begin{proof}
For simplicity, we set $x_{0}=0$ and $r=1$. By Lemma \ref{sol adjoint}, the map $T: G\mapsto v$ is a bounded linear operator on $L^{p}(B_{1/2})$. As before, we take $c=24$. For $\bar{y}\in B_{1/2}$ and $r\in (0,1/4)$, let $B=(B^{ij})_{i,j=1}^{n}\in L^{p}(B_{1})$ be a matrix-valued function supported in $B_{r}(\bar{y})\cap B_{1/2}$ with mean zero, and
\begin{align*}
b^{ij}=B^{ij}+\frac{\tilde{a}^{ij}}{\tilde{a}^{nn}}B^{nn},\quad(i,j)\neq(n,n),\quad b^{nn}=B^{nn}.
\end{align*}
By Lemma \ref{sol adjoint}, there exists an adjoint solution $v_{1}\in L^{p}(B_{1})$ of the problem
\begin{align*}
\begin{cases}
\tilde{L}_{x'_{0}}^{*}v_{1}=\Div^{2}b&\quad\mbox{in}~B_{1},\\
v_{1}=0&\quad\mbox{on}~\partial B_{1},
\end{cases}
\end{align*}
For any $R\in [cr,1]$ such that $B_{1/2}\setminus B_{R}(\bar{y})\neq\emptyset$ and $f\in C_{0}^{\infty}((B_{2R}(\bar{y})\setminus B_{R}(\bar{y}))\cap B_{1/2})$, let $v_{2}\in W_{0}^{2,p'}(B_{1})$ be a strong solution of
\begin{align*}
\begin{cases}
\tilde{L}_{x'_{0}}v_{2}=f&\quad\mbox{in}~B_{1},\\
v_{2}=0&\quad\mbox{on}~\partial B_{1}.
\end{cases}
\end{align*}
By using Definition \ref{def adjoint},
$$
D_{nn}v_{2}=\frac{1}{\tilde{a}^{nn}}\Big(f-\sum_{(i,j)\neq(n,n)}\tilde{a}^{ij}D_{ij}v_{2}\Big),
$$
the matrix $B$ is supported in $B_{r}(\bar{y})\cap B_{1/2}$ with mean zero, and $f=0$ in $B_{R/2}(\bar{y})$, we have
\begin{align}\label{iden v1 f}
&\int_{(B_{2R}(\bar{y})\setminus B_{R}(\bar{y}))\cap B_{1/2}}v_{1}f=\int_{B_{1}}D_{ij}v_{2}b^{ij}=\int_{B_{r}(\bar{y})\cap B_{1/2}}\sum_{(i,j)\neq(n,n)}D_{ij}v_{2}B^{ij}\nonumber\\
&=\int_{B_{1/2}\cap B_{r}(\bar{y})}\sum_{(i,j)\neq(n,n)}\big(D_{ij}v_{2}
-D_{ij}v_{2}(\bar y)\big)B^{ij}.
\end{align}
Recalling that in $B_{r}(\bar{y})\subset B_{R/24}(\bar{y})\subset B_{2/3}$, $\tilde{a}^{ij}(x)=\bar{a}^{ij}(x'_{0},x^{n})$, we see that $v_{2}$ satisfies
$$
\bar{a}^{ij}(x'_{0},x^{n})D_{ij}v_{2}=0\quad\mbox{in}~B_{R/24}(\bar{y}).
$$
By using Lemma \ref{lemma DDDx'} with a suitable scaling, we have
\begin{align}\label{DDDx'}
&\|D^{2}D_{x'}v_{2}\|_{L^{\infty}(B_{1/2}\cap B_{r}(\bar{y}))}\leq CR^{-1-\frac{n}{p'}}\|D^{2}v_{2}\|_{L^{p'}(B_{R/24}(\bar{y}))}\nonumber\\
&\leq CR^{-1-\frac{n}{p'}}\|D^{2}v_{2}\|_{L^{p'}(B_{1})}\leq CR^{-1-\frac{n}{p'}}\|f\|_{L^{p'}((B_{2R}(\bar{y})\setminus B_{R}(\bar{y}))\cap B_{1/2})}.
\end{align}
Coming back to \eqref{iden v1 f}, we use \eqref{DDDx'} to get
\begin{align*}
\left|\int_{(B_{2R}(\bar{y})\setminus B_{R}(\bar{y}))\cap B_{1/2}}v_{1}f\right|\leq CrR^{-1-\frac{n}{p'}}\|B\|_{L^{1}(B_{1/2}\cap B_{r}(\bar{y}))}\|f\|_{L^{p'}((B_{2R}(\bar{y})\setminus B_{R}(\bar{y}))\cap B_{1/2})}.
\end{align*}
The rest of the proof is identical to that of Lemma \ref{weak est v} and thus omitted.
\end{proof}

The following lemma is an analogy of Lemma \ref{lemma itera}. Set
$$\phi(x_{0},r):=\inf_{q_{0}\in\mathbb R}\left(\fint_{B_{r}(x_{0})}|\bar{u}-q_{0}|^{q}\ dx\right)^{1/q},$$
where $\bar{u}(x)=a^{nn}(x)u(x)-g^{nn}(x)$. We recall that the coordinate system is chosen according to each $x_0$.
\begin{lemma}\label{lemma itera000}
For any $\gamma\in (0,1)$ and $0<\rho\leq r\leq 1/4$, we have
\begin{align}\label{estimate phi''}
\phi(x_{0},\rho)\leq C\Big(\frac{\rho}{r}\Big)^{\gamma}r^{-n}\|\bar{u}\|_{L^{1}(B_{r}(x_{0}))}+C\tilde{\omega}_{A}(\rho)\big(\|\bar{u}\|_{L^{\infty}(B_{r}(x_{0}))}+\|g\|_{L^{\infty}(B_{r}(x_{0}))}\big)
+C\tilde{\omega}_{g}(\rho),
\end{align}
where $C=C(n,p,\delta,\gamma)>0$, and $\tilde\omega_{\bullet}(t)$ is a Dini function derived from $\omega_{\bullet}(t)$.
\end{lemma}
\begin{proof}
By Lemma \ref{lemma adjoint} with $G=(\tilde{A}(x)-A(x))u+g(x)-\bar{g}(x'_{0},x^{n})$ and the argument that led to \eqref{weak type DDv}, we have
\begin{equation*}
|\{x\in B_{r/2}(x_{0}): |v(x)|>t\}|\leq\frac{C}{t}\Big(r^{n}\bar\omega_{g}(r)+r^{n}\bar\omega_{A}(r)\|u\|_{L^{\infty}(B_{r}(x_{0}))}\Big).
\end{equation*}
Therefore, for any given $q\in (0,1)$, we have
\begin{equation}\label{estimate v}
\left(\fint_{B_{r/2}(x_{0})}|v|^{q}\ dx\right)^{1/q}\leq C\Big(\bar\omega_{g}(r)+\bar\omega_{A}(r)\|u\|_{L^{\infty}(B_{r}(x_{0}))}\Big).
\end{equation}
Let $w=u-v\in L^{p}(B_{r}(x_{0}))$, which satisfies $\bar{L}_{x'_{0}}^{*}w=\Div^{2}\bar{g}(x'_{0},x^{n})$ in $B_{r/2}(x_{0})$. Denote
$$\hat{a}^{ij}(x'_{0},x^{n}):=\frac{\bar{a}^{ij}(x'_{0},x^{n})}{\bar{a}^{nn}(x'_{0},x^{n})},\quad\bar{w}:=\bar{a}^{nn}(x'_{0},x^{n})w-\bar{g}^{nn}(x'_{0},x^{n}).$$
Then $\hat{a}^{nn}(x'_{0},x^{n})=1$, and
\begin{align*}
D_{ij}(\hat{a}^{ij}(x'_{0},x^{n})\bar{w})&=D_{ij}\left(\bar{a}^{ij}(x'_{0},x^{n})w-\hat{a}^{ij}(x'_{0},x^{n})\bar{g}^{nn}(x'_{0},x^{n})\right)\\
&=D_{ij}(\bar{a}^{ij}(x'_{0},x^{n})w)-\Div^{2}\bar{g}(x'_{0},x^{n})=0\quad\mbox{in}~ B_{r/2}(x_{0}).
\end{align*}
Now we prove that $D\bar w\in L_{\text{loc}}^{p}(B_{r/2}(x_{0}))$.
We choose $\eta\in C_{0}^{\infty}(B_{r/4}(x_{0}))$ with
$$0\leq\eta\leq1,\quad\eta\equiv1~\mbox{in}~B_{r/5}(x_{0}),
\quad|D\eta|\leq {40}/{r}.$$
Then
\begin{equation}\label{eq Dij}
D_{ij}(\hat{a}^{ij}(x'_{0},x^{n})\bar{w}\eta)=2D_{i}(\hat{a}^{ij}(x'_{0},x^{n})\bar{w}D_{j}\eta)-\hat{a}^{ij}(x'_{0},x^{n})\bar{w}D_{ij}\eta.
\end{equation}
Next we show that $\bar{w}\in W^{1,p}(B_{r/5}(x_{0}))$. For $k=1,\dots,n-1$ and $0<|h|<{r}/{12}$,
we take the finite difference quotient on both sides of \eqref{eq Dij} to get
\begin{align}\label{eq Dij2}
D_{ij}(\hat{a}^{ij}(x'_{0},x^{n})\delta_{h,k}(\bar{w}\eta))=2\delta_{h,k}D_{i}(\hat{a}^{ij}(x'_{0},x^{n})\bar{w}D_{j}\eta)-\hat{a}^{ij}(x'_{0},x^{n})\delta_{h,k}(\bar{w}D_{ij}\eta)
\end{align}
for any $x\in B_{r/3}(x_0)$.
For the first term of right-hand side in \eqref{eq Dij2}, we have $\hat{a}^{ij}(x'_{0},x^{n})\bar{w}D_{j}\eta\in L^{p}(B_{r/2}(x_{0}))$. For the second term of right-hand side in \eqref{eq Dij2}, we consider
\begin{align*}
\begin{cases}
\Delta V=-\hat{a}^{ij}(x'_{0},x^{n})\delta_{h,k}(\bar{w}D_{ij}\eta)&\quad\mbox{in}~B_{r/3}(x_{0}),\\
V=0&\quad\mbox{on}~\partial B_{r/3}(x_{0}).
\end{cases}
\end{align*}
We temporarily suppose that $\hat{a}^{ij}(x'_{0},x^{n})\bar{w}D_{ij}\eta$ is smooth. Then for each $x\in B_{r/3}(x_{0})$, $k=1,\dots,n-1$ and $0<|h|<{r}/{12}$, we have
\begin{align*}
&\hat{a}^{ij}(x'_{0},x^{n})\delta_{h,k}(\bar{w}D_{ij}\eta)
=\delta_{h,k}(\hat{a}^{ij}(x'_{0},x^{n})\bar{w}D_{ij}\eta)\\
&=\int_{0}^{1}D_{k}(\hat{a}^{ij}(x'_{0},x^{n})\bar{w}(x+the_{k})D_{ij}\eta(x+the_{k}))dt\cdot e_{k}\\
&=D_{k}\left(\int_{0}^{1}\hat{a}^{ij}(x'_{0},x^{n})\bar{w}(x+the_{k})D_{ij}\eta(x+the_{k})dt\cdot e_{k}\right).
\end{align*}
By using the $W^{1,p}$ estimate for the Poisson equation, we have
$$\|V\|_{W^{1,p}(B_{r/3}(x_{0}))}\leq C\|\bar{w}\|_{L^{p}(B_{r/2}(x_{0}))}.$$
Coming back to \eqref{eq Dij2}, we use the local estimate for the adjoint operator to get
$$\|\delta_{h,k}(\bar{w}\eta)\|_{L^{p}(B_{r/4}(x_{0}))}\leq C\|\bar{w}\|_{L^{p}(B_{r/2}(x_{0}))}.$$
This estimate holds if $\hat{a}^{ij}(x'_{0},x^{n})\bar{w}D_{ij}\eta$ is smooth, and thus is valid by approximation for $\hat{a}^{ij}(x'_{0},x^{n})\bar{w}D_{ij}\eta\in L^{p}(B_{r/3}(x_{0}))$. Therefore, we let $h\rightarrow0$ to obtain
$$\|D_{x'}\bar{w}\|_{L^{p}(B_{r/5}(x_{0}))}\leq C\|\bar{w}\|_{L^{p}(B_{r/2}(x_{0}))}.$$
Similarly, we have for any $k\geq1$, $D_{x'}^{k}\bar{w}\in L^{p}(B_{r/5}(x_{0}))$. By using the fact that $\breve{a}^{in}(x'_{0},x^{n})=0$ for $i=1,\ldots,n-1$, and $\breve{a}^{nn}(x'_{0},x^{n})=1$ (cf. \eqref{tilde a}), we have
$$D_{nn}\bar{w}=-\sum_{i,j=1}^{n-1}D_{i}\big(\breve{a}^{ij}(x'_{0},x^{n})D_{j}\bar{w}\big)
-\sum_{j=1}^{n-1}D_{n}\big(\breve{a}^{nj}(x'_{0},x^{n})D_{j}\bar{w}\big)\quad\mbox{in}~ B_{r/5}(x_{0}).$$
We now apply \cite[Corollary 4.4]{dk1} to conclude that $D_{n}\bar{w}\in L^{p}(B_{r/5}(x_{0}))$ and $\bar{w}\in W^{1,p}(B_{r/5}(x_{0}))$. Therefore, by repeating the same line that led to \eqref{trans}, we have
$$D_{i}(\breve{a}^{ij}(x'_{0},x^{n})D_{j}\bar{w})=D_{ij}(\hat{a}^{ij}(x'_{0},x^{n})\bar{w})=0\quad\mbox{in}~ B_{r/5}(x_{0}).$$
For any $q_{0}\in\mathbb R$,
by Lemma \ref{div xn} with a suitable scaling, we have
$$\|D\bar{w}\|_{L^{\infty}(B_{r/6}(x_{0}))}^{q}\leq Cr^{-(n+q)}\int_{B_{r/5}(x_{0})}|\bar{w}-q_{0}|^{q}\ dx.$$
Thus, similar to \eqref{ineq DDw}, we obtain
\begin{align*}
\left(\fint_{B_{\kappa r}(x_{0})}|\bar{w}-(\bar{w})_{B_{\kappa r}(x_{0})}|^{q}\ dx\right)^{1/q}\leq C_{0}\kappa\left(\fint_{B_{r/2}(x_{0})}|\bar{w}-q_{0}|^{q}\ dx\right)^{1/q}.
\end{align*}
Then by using \eqref{estimate v}, we get
\begin{align*}
&\left(\fint_{B_{\kappa r}(x_{0})}|\bar{u}-(\bar{w})_{B_{\kappa r}(x_{0})}|^{q}\ dx\right)^{1/q}\\
&\leq2^{1/q-1}\left(\fint_{B_{\kappa r}(x_{0})}|\bar{w}-(\bar{w})_{B_{\kappa r}(x_{0})}|^{q}\ dx\right)^{1/q}+C\Bigg(\fint_{B_{\kappa r}(x_{0})}|(a^{nn}(x)-\bar{a}^{nn}(x'_{0},x^{n}))u\\
&\quad+\bar{g}^{nn}(x'_{0},x^{n})-g^{nn}(x)+\bar{a}^{nn}(x'_{0},x^{n})v|^{q}\ dx\Bigg)^{1/q}\\
&\leq C_{0}\kappa\left(\fint_{B_{r/2}(x_{0})}|\bar{u}-q_{0}|^{q}\ dx\right)^{1/q}+C\kappa^{-{n}/{q}}\Big(\bar\omega_{A}(r)\|u\|_{L^{\infty}(B_{r}(x_{0}))}+\bar\omega_{g}(r)\Big).
\end{align*}
Therefore, similar to the argument that led to \eqref{phi DDu}, we have
\begin{align*}
\phi(x_{0},\kappa^{j}r)&\leq\kappa^{j\gamma}\phi(x_{0},r)+C\|u\|_{L^{\infty}(B_{r}(x_{0}))}\tilde\omega_{A}(\kappa^{j}r)+C\tilde\omega_{g}(\kappa^{j}r)\\
&\leq\kappa^{j\gamma}\phi(x_{0},r)+C\big(\|\bar{u}\|_{L^{\infty}(B_{r}(x_{0}))}+\|g\|_{L^{\infty}(B_{r}(x_{0}))}\big)\tilde\omega_{A}(\kappa^{j}r)+C\tilde\omega_{g}(\kappa^{j}r),
\end{align*}
which implies \eqref{estimate phi''}. The lemma is proved.
\end{proof}

Finally, we give the proof of Theorem \ref{thm C0}.
\begin{proof}[\bf Proof of Theorem \ref{thm C0}.]
We use the same argument that led to proof of Theorem \ref{thm} and list the main differences. By Lemma \ref{lemma itera000}, for any $r\in (0,1/8)$, we have
\begin{align*}
\sup_{x_{0}\in B_{1/8}}\phi(x_{0},r)\leq C\left(r^{\gamma}\|\bar{u}\|_{L^{1}(B_{1/4})}+\tilde{\omega}_{A}(r)\big(\|\bar{u}\|_{L^{\infty}(B_{1/4})}+\|g\|_{L^{\infty}(B_{1/4})}\big)+\tilde{\omega}_{g}(r)\right).
\end{align*}
Similar to \eqref{sup Dxx'}, we have
\begin{align*}
&\sup_{x_{0}\in B_{1/8}}|\bar{u}(x_{0})-q_{x_{0},r}|\nonumber\\
&\leq C\sup_{x_{0}\in B_{1/8}}\phi(x_{0},r)+C\big(\|\bar{u}\|_{L^{\infty}(B_{1/4})}+\|g\|_{L^{\infty}(B_{1/4})}\big)\int_{0}^{r}\frac{\tilde\omega_{A}(t)}{t}dt+C\int_{0}^{r}\frac{\tilde\omega_{g}(t)}{t}dt,
\end{align*}
where $q_{x_{0},r}\in\mathbb R$ satisfying
$$\phi(x_{0},r)=\left(\fint_{B_{r}(x_{0})}|\bar{u}-q_{x_{0},r}|^{q}\ dx\right)^{1/q}.$$
By repeating the same line of the proof of \eqref{est DDu''}, we have
\begin{align}\label{bound bar u}
\|\bar{u}\|_{L^{\infty}(B_{1/4})}\leq C\|\bar{u}\|_{L^{1}(B_{3/4})}+C\left(\int_{0}^{1}\frac{\tilde\omega_{g}(t)}{t}dt+\|g\|_{L^{\infty}(B_{3/4})}\right).
\end{align}
Using a similar argument as in the proof of Theorem \ref{thm}, for any $y_{0}\in B_{1/8}\cap \cD_{j_{1}}, j_{1}\in[1,l+1]$, we have the following two cases:

{\bf Case 1.} If $|x_{0}-y_{0}|\leq 1/32$, we set $r=|x_{0}-y_{0}|$. Recalling the definition of $\bar{u}$, we see that $a^{nn}$ and $g^{nn}$ depend on the coordinate system. Under the coordinate system associated with $y_{0}$, we use the notation $\bar{\mathfrak{u}}$. Then, similar to \eqref{diff DDu}, we get
$$|\bar{\mathfrak{u}}(y_{0})-\bar{u}(y_{0})|\leq C\big(\|u\|_{L^{\infty}(B_{1/4})}+\|g\|_{L^{\infty}(B_{1/4})}\big)\omega_{1}(|x_{0}-y_{0}|).$$
Thus, we obtain
\begin{align}\label{C est1}
&|\bar{u}(x_{0})-\bar{u}(y_{0})|\leq C\int_{0}^{|x_{0}-y_{0}|}\frac{\tilde{\omega}_{A}(t)}{t}dt\cdot\Bigg(\|u\|_{L^{1}(B_{3/4})}+\int_{0}^{1}\frac{\tilde{\omega}_{g}(t)}{t}dt+\|g\|_{L^{\infty}(B_{1})}\Bigg)\nonumber\\
&\quad+C|x_{0}-y_{0}|^{\gamma}\|u\|_{L^{1}(B_{1/4})}+C\int_{0}^{|x_{0}-y_{0}|}\frac{\tilde{\omega}_{g}(t)}{t}dt.
\end{align}

{\bf Case 2.} If $|x_{0}-y_{0}|\geq 1/32$, then
\begin{align}\label{C est2}
|\bar{u}(x_{0})-\bar{u}(y_{0})|\leq
C|x_{0}-y_{0}|^{\gamma}\left(\|u\|_{L^{1}(B_{3/4})}+\int_{0}^{1}\frac{\tilde\omega_{g}(t)}{t}dt+\|g\|_{L^{\infty}(B_{3/4})}\right).
\end{align}
The theorem is proved when $b^{i}\equiv c\equiv0$.

For the general case, we rewrite the equation as
$$D_{ij}(a^{ij}u)=\Div^{2}g+D_{i}(b^{i}u)-cu.$$
Consider
\begin{align*}
\begin{cases}
\Delta w=D_{i}(b^{i}u)-cu&\quad\mbox{in}~B_{1},\\
w=0&\quad\mbox{on}~\partial B_{1}.
\end{cases}
\end{align*}
Then, by the $W^{1,p}$ estimate, we have
\begin{align}\label{est w2}
\|w\|_{W^{1,p}(B_{1})}\leq C\|u\|_{L^{p}(B_{1})}.
\end{align}
Hence, we get
$$D_{ij}(a^{ij}u)=\Div^{2}(g+wI).$$
Then by using the local estimate for the adjoint operator and \eqref{est w2}, we have
\begin{align*}
\|u\|_{L^{p^{*}}(B_{1/2})}\leq C\left(\|g+w\|_{L^{p^{*}}(B_{1})}+\|u\|_{L^{p}(B_{1})}\right)\leq C\big(\|g\|_{L^{\infty}(B_{1})}+\|u\|_{L^{p}(B_{1})}\big),
\end{align*}
where ${1}/{p^{*}}={1}/{p}-{1}/{n}$ if $p<n$ and $p^{*}\in(p,\infty)$ is arbitrary if $p\geq n$. By a bootstrap argument, for any $q\in (1,\infty)$ we have $w\in W_{\text{loc}}^{1,q}(B_{1})$ and
$$\|w\|_{W^{1,q}(B_{1/2})}\leq C\big(\|g\|_{L^{\infty}(B_{1})}+\|u\|_{L^{p}(B_{1})}\big).$$
By Morrey's inequality, we can take a sufficiently large $q>n$ such that $w\in C_{\text{loc}}^{\beta}(B_{1})$ with $\beta=1-{n}/{q}>\max(\gamma,\mu/ (1+\mu))$, and
$$\|w\|_{C^{\beta}(B_{1/2})}\leq C\big(\|g\|_{L^{\infty}(B_{1})}+\|u\|_{L^{p}(B_{1})}\big).$$
Denote $g':=g+w$, and we have
$$\omega_{g'}(r)\leq \omega_{g}(r)+r^{\beta}[w]_{\beta;B_{1/2}}.$$
We can replace $g$ with $g'$ in \eqref{C est1} and \eqref{C est2}, respectively, to get
\begin{align*}
&|\bar{u}(x_{0})-\bar{u}(y_{0})|\nonumber\\
&\leq C\int_{0}^{|x_{0}-y_{0}|}\frac{\tilde{\omega}_{A}(t)}{t}dt\cdot\Bigg(\|u\|_{L^{p}(B_{3/4})}+\int_{0}^{1}\frac{\tilde{\omega}_{g}(t)}{t}dt+\|g\|_{L^{\infty}(B_{1})}\Bigg)\nonumber\\
&\quad+C|x_{0}-y_{0}|^{\gamma}
\left(\|g\|_{L^{\infty}(B_{1})}+\|u\|_{L^{p}(B_{1})}\right)+C\int_{0}^{|x_{0}-y_{0}|}\frac{\tilde{\omega}_{g}(t)}{t}dt,
\end{align*}
and
\begin{align*}
|\bar{u}(x_{0})-\bar{u}(y_{0})|\leq
C|x_{0}-y_{0}|^{\gamma}\left(\int_{0}^{1}\frac{\tilde\omega_{g}(t)}{t}dt+\|g\|_{L^{\infty}(B_{1})}+\|u\|_{L^{p}(B_{1})}\right).
\end{align*}
Theorem \ref{thm C0} is proved.
\end{proof}

Corollary \ref{coro u} follows from \eqref{modulus u} by using Lemma \ref{difference a holder} and taking $\gamma\in(\alpha,1)$.

\section{Proof of Corollary \ref{thm W21}}\label{pf of thm W21}

In this section, we will use the idea in \cite{MR2465684,em} to prove that if $u\in W^{2,1}(\cD)$ verifies $Lu=f$ {\em a.e.} in $\cD$ with $f\in L^{p}(\cD)$ for some $p\in(1,\infty)$, then $u\in W_{\text{loc}}^{2,p}(\cD)$.

\begin{proof}[\bf Proof of Corollary \ref{thm W21}]
We rewrite the equation \eqref{equations} as
$$Lu-\lambda_{0}u=f-\lambda_{0}u=:f_{0},$$
where $\lambda_{0}$ is a large fixed constant and $f_{0}\in L^{p}(\cD)$. Without loss of generality, we may assume that $1<p<{n}/(n-1)$. Let $\zeta\in C_{0}^{\infty}(\cD_{\varepsilon})$ with $\zeta\equiv1$ in $\cD'\subset\subset \cD_{\varepsilon}$, and $0\leq\zeta\leq1$. Let $\varphi\in C_{0}^{\infty}(\cD_{\varepsilon})$. We shall show that
\begin{align*}
\left|\int_{\cD_{\varepsilon}}D_{ij}(u\zeta)\varphi^{ij} \ dx\right|\leq C\|\varphi\|_{L^{p'}(\cD_{\varepsilon})}\big(\|f\|_{L^{p}(\cD)}+\|u\|_{W^{2,1}(\cD_{\varepsilon})}\big),
\end{align*}
where ${1}/{p}+{1}/{p'}=1$. Let $u_{\sigma}\in C^{\infty}(\cD_{\varepsilon})$ be a sequence of functions converging to $u$ in $W_{\text{loc}}^{2,1}(\cD_{\varepsilon})$ as $\sigma\rightarrow0$. Then for any $\varphi\in C_{0}^{\infty}(\cD_{\varepsilon})$, we have
\begin{align}\label{converg u sigma}
\int_{\cD_{\varepsilon}}D_{ij}(u\zeta)\varphi^{ij} \ dx=\lim_{\sigma\rightarrow0}\int_{\cD_{\varepsilon}}D_{ij}(u_{\sigma}\zeta)\varphi^{ij} \ dx.
\end{align}
By using the same idea that led to Lemma \ref{sol adjoint}, we modify the coefficients $a^{ij}$ to get
$$\tilde{a}^{ij}(x)=\eta a^{ij}(x)+\delta(1-\eta)\delta_{ij},$$
where $\eta\in C_{0}^{\infty}(\cD)$ is a cut-off function satisfying
$$0\leq\eta\leq1,\quad \eta\equiv1~\mbox{in}~\cD_{\varepsilon},\quad|\nabla\eta|\leq C(n,\varepsilon).$$
Then by using Lemma \ref{sol adjoint}, there exists an adjoint solution $v\in L^{p'}(\cD)$ to
\begin{align*}
\begin{cases}
D_{ij}(\tilde{a}^{ij}v)-D_{i}(b^{i}v)+(c-\lambda_{0})v=D_{ij}\varphi^{ij}&\quad\mbox{in}~\cD,\\
v=0&\quad\mbox{on}~\partial \cD,
\end{cases}
\end{align*}
and
\begin{align}\label{est v}
\|v\|_{L^{p'}(\cD)}\leq C\|\varphi\|_{L^{p'}(\cD_{\varepsilon})}.
\end{align}
Therefore, for any $w\in W^{2,p}(\cD)\cap W_{0}^{1,p}(\cD)$, we have
\begin{align*}
\int_{\cD}v\big(\tilde{a}^{ij}D_{ij}w+b^{i}D_{i}w+(c-\lambda_{0})w\big)\ dx=\int_{\cD}\varphi^{ij} D_{ij}w\ dx.
\end{align*}
It is easy to see that $u_{\sigma}\zeta\in W^{2,p}(\cD)\cap W_{0}^{1,p}(\cD)$ for any $\sigma>0$. Then,
\begin{align}\label{equa u sigma}
\int_{\cD}v\big(a^{ij}D_{ij}(u_{\sigma}\zeta)+b^{i}D_{i}(u_{\sigma}\zeta)+(c-\lambda_{0})(u_{\sigma}\zeta)\big)\ dx=\int_{\cD}\varphi^{ij} D_{ij}(u_{\sigma}\zeta)\ dx.
\end{align}
It follows from Theorem \ref{thm C0} that $v\in L^{\infty}(\cD_{\varepsilon})$. Since $u_{\sigma}\rightarrow u$ in $W_{\text{loc}}^{2,1}(\cD_{\varepsilon})$ as $\sigma\rightarrow0$, we thus use \eqref{converg u sigma} and \eqref{equa u sigma} to get
\begin{align*}
&\int_{\cD_{\varepsilon}}\varphi^{ij} D_{ij}(u\zeta)\ dx=\int_{\cD_{\varepsilon}}v\big(a^{ij}D_{ij}(u\zeta)+b^{i}D_{i}(u\zeta)+(c-\lambda_{0})u\zeta\big)\ dx\\
&=\int_{\cD_{\varepsilon}}v\big(a^{ij}D_{ij}u+b^{i}D_{i}u+(c-\lambda_{0})u\big)\zeta \ dx\\
&\quad+\int_{\cD_{\varepsilon}}\big(va^{ij}D_{ij}\zeta u+uvb^{i}D_{i}\zeta\big)\ dx+2\int_{\cD_{\varepsilon}}va^{ij}D_{i}uD_{j}\zeta \ dx.
\end{align*}
By using \eqref{est v}, we obtain
\begin{align*}
&\left|\int_{\cD_{\varepsilon}}v\big(a^{ij}D_{ij}u+b^{i}D_{i}u+(c-\lambda_{0})u\big)\zeta \ dx\right|\\
&\leq C\|v\|_{L^{p'}(\cD_{\varepsilon})}\|f_{0}\|_{L^{p}(\cD_{\varepsilon})}\leq C\|\varphi\|_{L^{p'}(\cD_{\varepsilon})}\big(\|f\|_{L^{p}(\cD_{\varepsilon})}+\|u\|_{W^{1,1}(\cD_{\varepsilon})}\big),
\end{align*}
\begin{align*}
\left|\int_{\cD_{\varepsilon}}\big(va^{ij}D_{ij}\zeta u+uvb^{i}D_{i}\zeta\big)\ dx\right|\leq C\|v\|_{L^{p'}(\cD_{\varepsilon})}\|u\|_{L^{p}(\cD_{\varepsilon})}\leq C \|\varphi\|_{L^{p'}(\cD_{\varepsilon})}\|u\|_{W^{1,1}(\cD_{\varepsilon})},
\end{align*}
and
\begin{align*}
\left|\int_{\cD_{\varepsilon}}va^{ij}D_{i}uD_{j}\zeta \ dx\right|\leq C\|v\|_{L^{p'}(\cD_{\varepsilon})}\|Du\|_{L^{p}(\cD_{\varepsilon})}\leq C \|\varphi\|_{L^{p'}(\cD_{\varepsilon})}\|u\|_{W^{2,1}(\cD_{\varepsilon})}.
\end{align*}
Therefore, we get
$$\left|\int_{\cD_{\varepsilon}}\varphi^{ij} D_{ij}(u\zeta)\ dx\right|\leq C\|\varphi\|_{L^{p'}(\cD_{\varepsilon})}\big( \|u\|_{W^{2,1}(\cD_{\varepsilon})}+\|f\|_{L^{p}(\cD_{\varepsilon})}\big).$$
We thus have $u\in W^{2,p}(\cD')$, and
\begin{equation*}
\|u\|_{W^{2,p}(\cD')}\leq C\big( \|u\|_{W^{2,1}(\cD_{\varepsilon})}+\|f\|_{L^{p}(\cD_{\varepsilon})}\big).
\end{equation*}
Corollary \ref{thm W21} is thus proved.
\end{proof}

\section{Weak-type $(1,1)$ estimates}\label{sec thm weak}

In this section, we consider the case when the sub-domains $\cD_{1},\dots,\cD_{M-1}$ are away from $\partial\cD$. In this case, we denote $\delta_{0}=\min_{1\leq j\leq M-1}\mbox{dist}\{\partial \cD_{j}, \partial\cD\}$. We derive global weak-type $(1,1)$ estimates with respect to an $A_{1}$ Muckenhoupt weight $w$ for solutions to the non-divergence form equation without lower-order terms and the corresponding adjoint problem. 
Denote
$$
w(\cD):=\int_{\cD}w(x)\ dx,\quad\|f\|_{L_{w}^{p}(\cD)}:=\left(\int_{\cD}|f|^{p}w\ dx\right)^{1/p},~ p\in[1,\infty),
$$
and
$$
W_{w}^{2,p}(\cD):=\{u: u, Du, D^{2}u\in L_{w}^{p}(\cD)\}.
$$

\begin{theorem}\label{thm weak C2}
Let $p\in (1,\infty)$, $\cD$ have a $C^{1,1}$ boundary, and $w$ be an $A_{1}$ Muckenhoupt weight. Suppose that the coefficients $A=(a^{ij})_{i,j=1}^{n}$ are of piecewise Dini mean oscillation over an open set containing $\overline{\cD}$. For $f\in L_{w}^{p}(\cD)$, let $u\in W_{w}^{2,p}(\cD)$ be a strong solution to
\begin{align*}
\begin{cases}
a^{ij}D_{ij}u=f&\quad\mbox{in}~\cD,\\
u=0&\quad\mbox{on}~\partial \cD.
\end{cases}
\end{align*}
Then for any $t>0$, we have
$$w\Big(\{x\in \cD: |D^{2}u(x)|>t\}\Big)\leq\frac{C}{t}\|f\|_{L_{w}^{1}(\cD)},$$
where $C$ depends on $n,M,p,\delta,\Lambda,\delta_{0}$, $[w]_{A_1}$, the $C^{1,1}$ norm of $\partial\cD$, and the $C^{1,\text{Dini}}$ characteristics of $\partial\cD_{j}$, $j=1,\dots,M-1$. Moreover, the linear operator $T: f\mapsto D^{2}u$ can be extended to a bounded operator from $L_{w}^{1}(\cD)$ to weak-$L_{w}^{1}(\cD)$.
\end{theorem}

\begin{remark}\label{rmk weak general}
From the proof below we can see that the result in Theorem \ref{thm weak C2} still holds for equations with lower-order terms  provided that $L1\leq0$ so that the weighted $W^{2,p}$ solvability is available; see Theorem \ref{solvability weight Lp}.
\end{remark}

To state the corresponding results for the adjoint operator, we need to impose additional conditions for the coefficient $A$ and the Dini function introduced in Definition \ref{def Dini}.
\begin{assumption}\label{assump omega}
1) $A$ is of piecewise Dini mean oscillation in $\cD$, and satisfies the following condition: there exists some constant $c_{0}>0$ such that for any $r\in(0,1/2)$, $\omega_{A}(r)\leq c_{0}(\ln r)^{-2}$.

2) For some constant $c_1,c_2>0$, $\omega_0'(R_0)\geq c_{1}$ and for any $R\in(0,R_0/2)$, $\omega_0(R)\leq c_{2}(\ln R)^{-2}$.
\end{assumption}

\begin{theorem}\label{thm weak C0}
Let $p\in (1,\infty)$, $\cD$ have a $C^{2,\text{Dini}}$ boundary, and $w$ be an $A_{1}$ Muckenhoupt weight. Under Assumption \ref{assump omega} the following hold. For $f=(f^{ij})_{i,j=1}^{n}\in L_{w}^{p}(\cD)$, let $u\in L_{w}^{p}(\cD)$ be a solution to the adjoint problem
\begin{align*}
\begin{cases}
D_{ij}(a^{ij}u)=\Div^{2}f&\quad\mbox{in}~\cD,\\
u=\frac{f\nu\cdot\nu}{A\nu\cdot\nu}&\quad\mbox{on}~\partial \cD.
\end{cases}
\end{align*}
Then for any $t>0$, we have
$$w\Big(\{x\in \cD: |u(x)|>t\}\Big)\leq\frac{C}{t}\|f\|_{L_{w}^{1}(\cD)},$$
where $C$ depends on $n,M,p,\delta,\Lambda,\delta_{0}$, $[w]_{A_1}$, the $C^{2,\text{Dini}}$ characteristics of $\partial\cD$ and the $C^{1,\text{Dini}}$ characteristics of $\partial\cD_{j}$, $j=1,\dots,M-1$. Moreover, the bounded linear operator $T: f\mapsto u$ can be extended to a bounded operator from $L_{w}^{1}(\cD)$ to weak-$L_{w}^{1}(\cD)$.
\end{theorem}

\begin{remark}
The result in Theorem \ref{thm weak C0} still holds for the problem
\begin{align*}
\begin{cases}
L^{*}u=\Div^{2}f&\quad\mbox{in}~\cD,\\
u=\frac{f\nu\cdot\nu}{A\nu\cdot\nu}&\quad\mbox{on}~\partial \cD,
\end{cases}
\end{align*}
when $L1\leq0$
by using Corollary \ref{solvability adjoint} and the same argument as in the proof of Theorem \ref{thm weak C0}.
\end{remark}

Instead of Lemma \ref{lemma weak II} above, which was used in \cite{dek,dk,dlk}, in the proofs of Theorems \ref{thm weak C2} and  \ref{thm weak C0}, we apply a generalized version of it stated below since our argument and estimates depend on the coordinate system associated with a given point, as mentioned before.

\begin{lemma}\label{general version}\cite[Lemma 6.3]{dx}
Let $w$ be a doubling measure and $\cD$ be a bounded domain in $\mathbb R^{n}$ satisfying \eqref{condition D}. Let $p\in (1,\infty)$ and $T$ be a bounded linear operator on $L_{w}^{p}(\cD)$. Suppose that if for some $f\in L_{w}^p(\cD)$, $t>0$, and some cube $Q_\alpha^k$ we have
$$
t<\frac{1}{w(Q_\alpha^k)}\int_{Q_\alpha^k}|f|w\ dx\leq C_{1}t,
$$
where $C_{1}\geq1$ and $\{Q_\alpha^k\}$ is a collection of ``cubes" defined in \cite[Appendix]{dx}, then $f$ admits a decomposition $f=g+b$ in $Q_\alpha^k$, where $g$ and $b$ satisfy
\begin{equation*}
\int_{Q_\alpha^k}|g|^{p}w\ dx\leq C_1t^{p}w(Q^k_\alpha),\quad\int_{\cD\setminus B_{cr}(x_{0})}|T(b\chi_{Q_\alpha^k})|w\ dx\leq C_1tw(Q^k_\alpha)
\end{equation*}
with $x_0\in Q_\alpha^k$ and $r=\text{diam}\,Q_\alpha^k$. Then for any $f\in L_{w}^p(\cD)$ and $t>0$, we have
\begin{equation*}
w(\{x\in \cD: |Tf(x)|>t\})\leq\frac{C}{t}\int_{\cD}|f|w\ dx,
\end{equation*}
where $C=C(n,c,\cD,C_1,\|T\|_{L_{w}^p\rightarrow L_{w}^p})$ is a constant. Moreover, $T$ can be extended to a bounded operator from $L_{w}^1(\cD)$ to weak-$L_{w}^1(\cD)$.
\end{lemma}

\begin{proof}[\bf Proof of Theorem \ref{thm weak C2}.]
By Theorem \ref{solvability weight Lp}, we can see that the map $T: f\mapsto D^{2}u$ is a bounded linear operator on $L_{w}^{p}(\cD)$. It suffices to show that $T$ satisfies the hypothesis of Lemma \ref{general version}. For simplicity, we may assume that $\cD$ is contained in $B_{5}$ and $A$ has piecewise Dini mean oscillation on $B_{10}$. Let $\{Q_{\alpha}^{k}\}$ be a
collection of dyadic ``cubes'' introduced in the proof of \cite[Lemma 4.1]{dek}. Notice that the assumptions in Lemma \ref{general version} satisfies automatically for large cubes (i.e., small $k$) by taking a sufficiently large $c$. We thus can assume that each $Q_{\alpha}^{k}$ is small enough so that they do not intersect with $\cup_{j=1}^{M-1}\overline{\cD_{j}}$ and $\partial\cD$ at the same time. The following proof proceeds in the same way as in \cite[Theorem 1.10]{dek} except that in our case, we consider the decomposition of $f$ with respect to the $A_{1}$ Muckenhoupt weight $w$; that is, for some $Q_{\alpha}^{k}$ and $t>0$, suppose
\begin{equation}\label{int f weight}
t<\frac{1}{w(Q_{\alpha}^{k})}\int_{Q_{\alpha}^{k}}|f|w\ dx\leq C_{1}t,
\end{equation}
where $C_{1}\geq1$. For a fixed $x_{k}\in Q_{\alpha}^{k}$, we associate $Q_{\alpha}^{k}$ with a Euclidean ball $B_{r_k}(x_k)$ such that $Q_{\alpha}^{k}\subset B_{r_k}(x_k)$, where $r_k:=\text{diam}\, Q_{\alpha}^{k}\leq {\delta_{0}}/{2}$. Let $W$ be the nonnegative adjoint solution to
$$D_{ij}(a^{ij}W)=0\quad\mbox{in}~B_{10},\quad W=1\quad\mbox{on}~\partial B_{10}.$$
Then $W$ is in the reverse H\"{o}lder class, with constants which depend only on $n,\delta$, and $\Lambda$:
$$(W)_{B_{2r_{k}}(x_{k})}\leq C(W)_{B_{r_k}(x_k)},\quad\left(\fint_{B_{r_k}(x_k)}W^{\frac{n}{n-1}}\ dx\right)^{\frac{n-1}{n}}\leq C\fint_{B_{r_k}(x_k)}W\ dx,$$
whenever $B_{2r_{k}}(x_{k})\subset B_{10}$. Also, $(W)_{B_{10}}\approx1$; see \cite{dek,e,fgms,fs}. Then we have the following global estimate: For any $x_{k}\in\cD$ and $0<r_{k}\leq \delta_0/2$,
by using \eqref{bound bar u} when $B_{k}\cap \partial\cD=\emptyset$ and \cite[Lemma 2.26]{dek} when $B_{k}\cap \partial\cD_{j}=\emptyset$, $j=1,\dots,M-1$, we get
\begin{equation}\label{estimate W}
||W||_{L^{\infty}(\cD_{r_{k}}(x_{k}))}\leq Cr_{k}^{-n}||W||_{L^{1}(\cD_{r_{k}}(x_{k}))},
\end{equation}
where $C$ depends on $n,M,\delta,\Lambda,\delta_{0},\omega_{A}$, the $C^{1,1}$ norm of $\partial\cD$, and 
the $C^{1,\text{Dini}}$ characteristics of $\partial\cD_{j}$, $j=1,\dots,M-1$. We now decompose $f=g+b$ in a given $Q_{\alpha}^{k}$ such that
$$
g:=\frac{1}{W(Q_{\alpha}^{k})}\int_{Q_{\alpha}^{k}}fW,
\quad b=f-g.$$
Then
\begin{equation}\label{prop bW}
\int_{Q_{\alpha}^{k}}bW\ dx=0,
\end{equation}
and by using \eqref{estimate W}, the definition of $A_1$ weights, and \eqref{int f weight}, we have
\begin{align*}
|g|\leq \frac{||W||_{L^{\infty}(Q_{\alpha}^{k})}}{W(Q_{\alpha}^{k})}\int_{Q_{\alpha}^{k}}|f|\ dx\leq \frac{C}{|Q_{\alpha}^{k}|\inf\limits_{Q_{\alpha}^{k}}w}\int_{Q_{\alpha}^{k}}|f|w\ dx\leq\frac{C}{w(Q_{\alpha}^{k})}\int_{Q_{\alpha}^{k}}|f|w\ dx\leq Ct.
\end{align*}
We thus get
$$\int_{Q_{\alpha}^{k}}|g|^{p}w\ dx\leq Ct^{p}w(Q_{\alpha}^{k}).$$
Let $u_{1}\in W_{w}^{2,p}$ be the unique solution to
\begin{equation*}
\begin{cases}
a^{ij}D_{ij}u_{1}=b \chi_{Q_{\alpha}^{k}}&\quad\mbox{in}~\cD,\\
u_{1}=0&\quad\mbox{on}~\partial\cD.
\end{cases}
\end{equation*}
Set $R_{0}=\mbox{diam}~\cD$ and $c=4R_{0}/\delta_0$. For any $R\in  [cr_{k},R_0)$ such that $\cD\setminus B_{R}(x_{k})\neq\emptyset$ and $h\in C_{0}^{\infty}(\cD_{2R}(x_{k})\setminus B_{R}(x_{k}))$, in view of Corollary \ref{solvability adjoint} below, there exists a unique adjoint solution $u_{2}\in L_{w^{-{1}/(p-1)}}^{p'}$ of
\begin{equation*}
\begin{cases}
D_{ij}(a^{ij}u_{2})=\Div^{2}h&\quad\mbox{in}~\cD,\\
u_{2}=0&\quad\mbox{on}~\partial\cD.
\end{cases}
\end{equation*}
Let $\tilde{u}_{2}:={u_{2}}/{W}$. Then by duality and \eqref{prop bW}, we have
$$
\int_{\cD}D_{ij}u_{1}h^{ij}=\int_{Q_{\alpha}^{k}}\tilde{u}_{2}Wb
=\int_{Q_{\alpha}^{k}}(\tilde{u}_{2}-\tilde{u}_{2}(x_{k}))Wb.
$$
Similar to \cite[(3.6)-(3.8)]{dek} (see also \cite{b,e,fgms}), we have
\begin{align*}
&\|\tilde{u}_{2}-\tilde{u}_{2}(x_{k})\|_{L^{\infty}(Q_{\alpha}^{k})}
\leq C\big(\frac{r_{k}}{R}\big)^{\alpha}
\|\tilde{u}_{2}\|_{L^{\infty}(\cD_{{\delta_{0}R}/(2R_{0})}(x_{k}))}\\
&\leq C\big(\frac{r_{k}}{R}\big)^{\alpha}
\frac{1}{W(B_{{\delta_{0}R}/(2R_{0})}(x_{k}))}\int_{\cD_{R}(x_{k})}|\tilde{u}_{2}|W\ dx\\
&=C\big(\frac{r_{k}}{R}\big)^{\alpha}\frac{1}{W(B_{{\delta_{0}R}/(2R_{0})}(x_{k}))}\int_{\cD_{R}(x_{k})}|u_{2}|,
\end{align*}
where $\alpha\in (0,1)$. Hence, we obtain
\begin{align*}
&\left|\int_{\cD_{2R}(x_{k})\setminus B_{R}(x_{k})}D_{ij}u_{1}h^{ij}\right|\leq C\big(\frac{r_{k}}{R}\big)^{\alpha}\frac{\|W\|_{L^{\infty}
(B_{{\delta_{0}R}/(2R_{0})}(x_{k}))}}
{W(B_{{\delta_{0}R}/(2R_{0})}(x_{k}))}\|u_{2}\|_{L^{1}(\cD_{R}(x_{k}))}\|b\|_{L^{1}(Q_{\alpha}^{k})}\\
&\leq C\big(\frac{r_{k}}{R}\big)^{\alpha}R^{-n}\|u_{2}\|_{L^{1}(\cD_{R}(x_{k}))}\|b\|_{L^{1}(Q_{\alpha}^{k})}\\
&\leq C\big(\frac{r_{k}}{R}\big)^{\alpha}R^{-n}\left(\int_{\cD}|u_{2}|^{p'}w^{-\frac{1}{p-1}}\ dx\right)^{\frac{1}{p'}}\big(\int_{\cD_{2R}(x_{k})}w\ dx\big)^{\frac{1}{p}}\frac{1}{\inf\limits_{Q_{\alpha}^{k}}w}\int_{Q_{\alpha}^{k}}|b|w\ dx\\
&\leq C\big(\frac{r_{k}}{R}\big)^{\alpha}\left(\int_{\cD_{2R}(x_{k})\setminus B_{R}(x_{k})}|h|^{p'}w^{-\frac{1}{p-1}}\ dx\right)^{\frac{1}{p'}}\big(\int_{\cD_{2R}(x_{k})}w\ dx\big)^{\frac{1}{p}-1}\int_{Q_{\alpha}^{k}}|b|w\ dx,
\end{align*}
where we used \eqref{estimate W}, H\"{o}lder's inequality, the definition of $A_{1}$ Muckenhoupt weights, and \eqref{adjoint es u}. By duality, we have
$$\|D^{2}u_{1}\|_{L_{w}^{p}(\cD_{2R}(x_{k})\setminus B_{R}(x_{k}))}\leq C\big(\frac{r_{k}}{R}\big)^{\alpha}\big(\int_{\cD_{2R}(x_{k})}w\ dx\big)^{\frac{1}{p}-1}\|b\|_{L_{w}^{1}(Q_{\alpha}^{k})}.$$
Therefore, by H\"{o}lder's inequality, we obtain
$$\|D^{2}u_{1}\|_{L_{w}^{1}(\cD_{2R}(x_{k})\setminus B_{R}(x_{k}))}\leq C\big(\frac{r_{k}}{R}\big)^{\alpha}\|b\|_{L_{w}^{1}(Q_{\alpha}^{k})}.$$
The rest of proof is identical to that of Lemma \ref{weak est v} and thus omitted. Hence, we get the desired result by using Lemma \ref{general version}.
\end{proof}

We will also use the following lemma.
\begin{lemma}[Lemma 3.4 of \cite{dk}]\label{lemma A}
Let $\omega$ be a nonnegative increasing function such that $\omega(t)\leq(\ln\frac{t}{4})^{-2}$ for $0<t\leq1$, and $\tilde\omega$ be given as in \eqref{tilde phi} with $\omega$ in place of $\bar\omega$. Then for any $r\in(0,1]$, we have
$$\int_{0}^{r}\frac{\tilde\omega(t)}{t}dt\leq C\Big(\ln\frac{4}{r}\Big)^{-1},$$
where $C>0$ is some positive constant.
\end{lemma}

\begin{proof}[\bf Proof of Theorem \ref{thm weak C0}.]
By Corollary \ref{solvability adjoint}, one can see that the map $T: f\mapsto u$ is a bounded linear operator on $L_{w}^{p}(\cD)$. We follow the argument in the proof of \cite[Theorem 5.2]{dx} with minor modifications. Under the same conditions that $f\in L_{w}^{p}(\cD)$ and $Q_{\alpha}^{k}$ as mentioned in the proof of Theorem \ref{thm weak C2}, we decompose $f$ in a given cube $Q_{\alpha}^{k}$ according to the following two cases.

(i) If $\dist(x_{k},\partial\cD)\leq \delta_{0}/2$, then $B_{r_{k}}(x_{k})\cap \cD_{j}=\emptyset$, $j=1,\dots,M-1$. In this case, we take $y_{k}\in\partial\cD$ such that $|x_{k}-y_{k}|=\dist(x_{k},\partial\cD)$. Let
$$
g:=\fint_{Q_{\alpha}^{k}}f\ dx,\quad b=f-g\quad\mbox{in}~Q_{\alpha}^{k}.
$$
Then $(b)_{Q_{\alpha}^{k}}=0$ and
\begin{align*}
|g|\leq \fint_{Q_{\alpha}^{k}}|f|\ dx\leq \frac{1}{|Q_{\alpha}^{k}|\inf\limits_{Q_{\alpha}^{k}}w}\int_{Q_{\alpha}^{k}}|f|w\ dx\leq\frac{C}{w(Q_{\alpha}^{k})}\int_{Q_{\alpha}^{k}}|f|w\ dx\leq Ct,
\end{align*}
where we used the definition of $A_1$ weights and \eqref{int f weight}. Hence,
\begin{equation*}
\int_{Q_{\alpha}^{k}}|g|^{p}w\ dx\leq Ct^{p}w(Q_{\alpha}^{k}).
\end{equation*}
We now check the hypothesis regarding $b$. Let $v_{1}\in L_{w}^{p}(\cD)$ be an adjoint solution of the problem
\begin{align*}
\begin{cases}
D_{ij}(a^{ij}v_{1})=\Div^{2}(b\chi_{Q_{\alpha}^{k}})&\quad\mbox{in}~\cD,\\
v_{1}={b\chi_{Q_{\alpha}^{k}}\nu\cdot\nu}
/(A\nu\cdot\nu)&\quad\mbox{on}~\partial\cD,
\end{cases}
\end{align*}
the solvability of which follows from Corollary \ref{solvability adjoint}. Set $R_0=\mbox{diam}~\cD$ and $c=4R_{0}/\delta_{0}$. Then for any $R\in [cr_{k},R_0)$ such that $\cD\setminus B_{R}(x_{k})\neq\emptyset$ and $h\in C_{0}^{\infty}(\cD_{2R}(x_{k})\setminus B_{R}(x_{k}))$, let $v_{2}\in W_{w^{-1/(p-1)}}^{2,p'}(\cD)$ be a strong solution of
\begin{align*}
\begin{cases}
a^{ij}D_{ij}v_{2}=h&\quad\mbox{in}~\cD,\\
v_{2}=0&\quad\mbox{on}~\partial\cD.
\end{cases}
\end{align*}
By using Definition \ref{def adjoint}, the matrix $b$ is supported in $Q_{\alpha}^{k}$ with mean zero, and $h\in C_{0}^{\infty}(\cD_{2R}(x_{k})\setminus B_{R}(x_{k}))$, we have
\begin{align}\label{identity v1 h}
\int_{\cD_{2R}(x_{k})\setminus B_{R}(x_{k})}v_{1}h&=\int_{Q_{\alpha}^{k}}D_{ij}v_{2}b^{ij}=\int_{Q_{\alpha}^{k}}\big(D_{ij}v_{2}
-D_{ij}v_{2}(x_{k})\big)b^{ij}.
\end{align}
Since $R\leq R_{0}$, $B_{{\delta_{0}R}/(2R_{0})}(x_{k})$ does not intersect with sub-domains. Also, $a^{ij}D_{ij}v_{2}=0$ in $\cD_{R}(x_{k})$. Then by flattening the boundary and using a similar argument that led to an a
priori estimate of the modulus of continuity of $D^{2}v_{2}$ in the proof of \cite[Theorem 1.5]{dek}, we have
\begin{align}\label{est boundary}
|D^{2}v_{2}(x)-D^{2}v_{2}(x_{k})|\leq C\left(\Big(\frac{|x-x_{k}|}{R}\Big)^{\gamma}+\omega_{A}^{*}(|x-x_{k}|)\right)R^{-n}\|D^{2}v_{2}\|_{L^{1}
(\cD_{{\delta_{0}R}/(2R_{0})}(x_{k}))}
\end{align}
for any $x\in Q_{\alpha}^{k}\subset\cD_{{\delta_{0}R}/(4R_{0})}(x_{k})$, where $\gamma\in(0,1)$ is a constant and for $0<t\leq1$,
$$
\omega_{A}^{*}(t):=\hat{\omega}_{A}(t)+\int_{0}^{t}\frac{\tilde{\omega}_{A}(s)}{s}\ ds+\tilde{\omega}_{A}(4t)+\int_{0}^{t}\frac{\tilde{\omega}_A(4s)}{s}\ ds,
$$
$$
\hat{\omega}_{A}(t)
:=\tilde{\omega}_{A}(t)+\tilde{\omega}_{A}(4t)+\omega_{A}^{\sharp}(4t),\quad
\text{and}\quad
\omega_{A}^{\sharp}(t)
:=\sup_{s\in[t,1]}(t/s)^{\gamma}\tilde{\omega}_{A}(s).
$$
Then, coming back to \eqref{identity v1 h}, we obtain
\begin{align*}
&\left|\int_{\cD_{2R}(x_{k})\setminus B_{R}(x_{k})}v_{1}h\right|\leq\frac{1}{\inf\limits_{Q_{\alpha}^{k}}w}\|b\|_{L_{w}^{1}(Q_{\alpha}^{k})}
\|D^{2}v_{2}-D^{2}v_{2}(x_{k})\|_{L^{\infty}(Q_{\alpha}^{k})}\\
&\leq\frac{CR^{-n}}{\inf\limits_{\cD_{2R}(x_{k})}w}\|b\|_{L_{w}^{1}(Q_{\alpha}^{k})}
\|D^{2}v_{2}\|_{L^{1}(\cD_{{\delta_{0}R}/(2R_{0})}(x_{k}))}\left(r_{k}^{\gamma}R^{-\gamma}+\Big(\ln\frac{4}{r_{k}}\Big)^{-1}\right)\\
&\leq C\big(\int_{\cD_{2R}(x_{k})}w\ dx\big)^{\frac{1}{p}-1}\|b\|_{L_{w}^{1}(Q_{\alpha}^{k})}\left(\int_{\cD}|D^{2}v_{2}|^{p'}w^{-\frac{1}{p-1}}\ dx\right)^{\frac{1}{p'}}\left(r_{k}^{\gamma}R^{-\gamma}+\Big(\ln\frac{4}{r_{k}}\Big)^{-1}\right)\\
&\leq C\big(\int_{\cD_{2R}(x_{k})}w\ dx\big)^{\frac{1}{p}-1}\|b\|_{L_{w}^{1}(Q_{\alpha}^{k})}\left(\int_{\cD_{2R}(x_{k})\setminus B_{R}(x_{k})}|h|^{p'}w^{-\frac{1}{p-1}}\ dx\right)^{\frac{1}{p'}}\left(r_{k}^{\gamma}R^{-\gamma}+\Big(\ln\frac{4}{r_{k}}\Big)^{-1}\right),
\end{align*}
where we used \eqref{est boundary}, Lemma \ref{lemma A}, the definition of $A_{1}$ Muckenhoupt weights, H\"{o}lder's inequality, and the estimate
$$\left(\int_{\cD}|D^{2}v_{2}|^{p'}w^{-\frac{1}{p-1}}\ dx\right)^{\frac{1}{p'}}\leq C\left(\int_{\cD}|h|^{p'}w^{-\frac{1}{p-1}}\ dx\right)^{\frac{1}{p'}}=C\left(\int_{\cD_{2R}(x_{k})\setminus B_{R}(x_{k})}|h|^{p'}w^{-\frac{1}{p-1}}\ dx\right)^{\frac{1}{p'}}.$$
The rest of proof is identical to that of Theorem \ref{thm weak C2}. We thus obtain
\begin{align*}
\int_{\cD\setminus B_{cr_{k}}(x_{k})}|v_{1}|w\ dx \leq C\int_{Q_{\alpha}^{k}}|b|w\ dx\leq C\int_{Q_{\alpha}^{k}}|f|w\ dx+C\int_{Q_{\alpha}^{k}}|g|w\ dx\leq Ctw(Q_{\alpha}^{k}).
\end{align*}
That is,
$$\int_{\cD\setminus B_{cr_{k}}(x_{k})}|Tb\chi_{Q_{\alpha}^{k}}|w\ dx\leq Ctw(Q_{\alpha}^{k}).$$

(ii) If $\dist(x_{k},\partial\cD)\geq \delta_{0}/2$, then $B_{r_{k}}(x_{k})\cap \partial\cD=\emptyset$. In this case, we choose the coordinate system according to $x_{k}$. In $Q_\alpha^k$, we set
\begin{align*}
g^{ij}=
\fint_{Q_{\alpha}^{k}}\big(f^{ij}-\frac{a^{ij}}{a^{nn}}f^{nn}\big)\ dx+\frac{a^{ij}}{a^{nn}}\fint_{Q_{\alpha}^{k}}f^{nn}\ dx,\,\,(i,j)\neq(n,n),\quad g^{nn}=\fint_{Q_{\alpha}^{k}}f^{nn}\ dx,
\end{align*}
and $b=f-g$. Then
\begin{equation*}
\int_{Q_{\alpha}^{k}}|g|^{p}w\ dx\leq Ct^{p}w(Q_{\alpha}^{k}).
\end{equation*}
Let
\begin{align*}
\tilde{b}^{ij}=b^{ij}-\frac{a^{ij}}{a^{nn}}b^{nn},\quad(i,j)\neq(n,n),\quad
\tilde{b}^{nn}=b^{nn}.
\end{align*}
Then $(\tilde{b})_{Q_{\alpha}^{k}}=0$.
It then follow from the argument as in the first case that
\begin{align}\label{iden v1 f11}
&\int_{\cD_{2R}(x_{k})\setminus B_{R}(x_{k})}v_{1}h=\int_{Q_{\alpha}^{k}}D_{ij}v_{2}b^{ij}\nonumber\\
&=\int_{Q_{\alpha}^{k}}\sum_{(i,j)\neq(n,n)}D_{ij}v_{2}\tilde{b}^{ij}
=\int_{Q_{\alpha}^{k}}\sum_{(i,j)\neq(n,n)}\big(D_{ij}v_{2}
-D_{ij}v_{2}(x_{k})\big)\tilde{b}^{ij},
\end{align}
where we used
$$
D_{nn}v_{2}=-\sum_{(i,j)\neq(n,n)}\frac{a^{ij}}{a^{nn}}D_{ij}v_{2}\quad\mbox{in}~B_{r_{k}}(x_{k}).
$$
Recalling that $cr_{k}\leq R\leq R_{0}$ so that $B_{{\delta_{0}R}/(2R_{0})}(x_{k})\cap \partial\cD=\emptyset$. Then by using a similar argument that led to \eqref{C2 est1} (or \eqref{C2 est}, \eqref{C2 est2}), we obtain
\begin{align*}
&|D_{xx'}v_{2}(x)-D_{xx'}v_{2}(x_{k})|\\
&\leq C R^{-n}\left(\Big(\frac{|x-x_{k}|}{R}\Big)^{\gamma}+\int_{0}^{|x- x_{k}|}\frac{\tilde{\omega}_{A}(t)}{t}dt\right)
\|D_{xx'}v_{2}\|_{L^{1}(B_{{\delta_{0}R}/(2R_{0})}(x_{k}))}
\end{align*}
for any $x\in Q_{\alpha}^{k}\subset B_{{\delta_{0}R}/(4R_{0})}(x_{k})$.
Coming back to \eqref{iden v1 f11} and using a similar argument as in the case (i), we obtain
\begin{align*}
&\int_{\cD\setminus B_{cr_{k}}(x_{k})}|v_{1}|w\ dx\leq C\int_{Q_{\alpha}^{k}}|\tilde{b}|w\ dx\\
&\leq C\int_{Q_{\alpha}^{k}}|b|w\ dx\leq C\int_{Q_{\alpha}^{k}}|f|w\ dx+C\int_{Q_{\alpha}^{k}}|g|w\ dx\leq Ctw(Q_{\alpha}^{k}).
\end{align*}
By Lemma \ref{general version}, the theorem is proved.
\end{proof}

\section{Appendix}

In the appendix, we give the $W_{w}^{2,p}$-estimate and solvability for the non-divergence form elliptic equation in $C^{1,1}$ domains with the zero Dirichlet boundary condition.
Consider
\begin{equation}\label{eq weighted}
\begin{cases}
\lambda u-Lu=f&\quad\mbox{in}~\cD,\\
u=0&\quad\mbox{on}~\partial\cD,
\end{cases}
\end{equation}
where $\lambda\geq0$, $\cD\in C^{1,1}$, and $Lu:=a^{ij}D_{ij}u+b^{i}D_{i}u+cu$. Let $p\in(1,\infty)$ and $w$ be an $A_p$ weight. Denote
$$
\mathring{W}_{w}^{2,p}(\cD):=\{u\in W_{w}^{2,p}(\cD): u=0~\mbox{on}~\partial\cD\}.
$$
Now we impose the regularity assumptions on $a^{ij}$. Let $\gamma_{0}=\gamma_{0}(n,p,\delta,[w]_{A_{p}})\in(0,1)$ be a sufficiently small constant to be specified. There exists a constant $r_{0}\in(0,1)$ such that $a^{ij}$ satisfy \eqref{condi A} in the interior of $\cD$ and are VMO near the boundary: for any $x_{0}\in\partial\cD$ and $r\in(0,r_{0}]$, we have
$$
\fint_{B_{r}(x_{0})}|a^{ij}(x)-(a^{ij})_{B_{r}(x_{0})}|\ dx\leq\gamma_{0}.
$$
In addition, $b^{i}$ and $c$ are bounded by a constant $\Lambda$. Then we have the following

\begin{theorem}\label{solvability weight Lp}
Let $p\in (1,\infty)$, $w\in A_p$, and $L1\leq0$. There exists a sufficient small constant $\gamma_0=\gamma_{0}(n,p,\delta,[w]_{A_{p}})\in(0,1)$ such that under the above conditions, for any $\lambda\geq0$ and $f\in L_{w}^{p}(\cD)$, there exists a unique $u\in W_{w}^{2,p}(\cD)$ satisfying \eqref{eq weighted}. Furthermore, there exists a constant $C=C(n,p,\delta,\Lambda,\cD,[w]_{A_{p}},r_{0})$ such that
$$\|u\|_{W_{w}^{2,p}(\cD)}\leq C\|f\|_{L_{w}^{p}(\cD)}.$$
\end{theorem}

First we note that when $\lambda\ge \lambda_1(n,p,\delta,\lambda, \cD, [w]_{A_{p}},r_{0})$, the theorem follows from the proofs of \cite[Theorems 6.3 and 6.4]{dk2} combined with the argument in \cite[Theorem 2.5]{d1} and \cite[Sections 8.5]{k}. To deal with the case when $\lambda\in [0,\lambda_1)$, we need the following lemmas.
The first one is a local regularity of solutions in weighted Sobolev spaces.

\begin{lemma}\label{appen lem regularity}
Let $1<p\leq q<\infty$, $z\in\overline{\cD}$. Denote $\cD_{r}:=\cD\cap B_{r}(z)$. Then if
\begin{equation}\label{condition}
\xi u\in\mathring{W}_{w}^{2,p}(\cD_{2R})\quad\forall~\xi\in C_{0}^{\infty}(B_{2R}(z)),\quad Lu\in L_{w}^{q}(\cD_{2R}),
\end{equation}
we have
\begin{equation}\label{regularity xi u}
\xi u\in\mathring{W}_{w}^{2,q}(\cD_{2R})\quad\forall~\xi\in C_{0}^{\infty}(B_{2R}(z)).
\end{equation}
Furthermore, there exists a constant $C=C(R,p,q,n,\delta,\Lambda,[w]_{A_{p}},r_{0})$ such that if \eqref{condition} holds, then
\begin{equation}\label{es u W2q}
\|u\|_{W_{w}^{2,q}(\cD_{R})}\leq C\big(\|Lu\|_{L_{w}^{q}(\cD_{2R})}+\|u\|_{L_{w}^{p}(\cD_{2R})}\big).
\end{equation}
\end{lemma}

\begin{proof}
We follow the proof of \cite[Theorem 11.2.3]{k} when $w=1$. For $q=p$, \eqref{regularity xi u} is obvious and \eqref{es u W2q} is obtained by using the method in the proof of \cite[Theorem 9.4.1]{k}. For $q>p$, we define
\begin{align*}
\alpha=\frac{n}{n-1}\quad\mbox{for}~n\geq2;\quad p(j)=\alpha^{j}p,\quad j=0,1,\dots,k-1,\quad p(k)=q,
\end{align*}
where $k-1$ is the last $j$ such that $p(j)<q$. Take $\lambda$ sufficiently large that $\lambda-L$, as an operator acting from $\mathring{W}_{w}^{2,p(j)}(\cD)$ onto $L_{w}^{p(j)}(\cD)$ for $j=0,1,\dots,k$,  is invertible. 
Denote
$$f=Lu,\quad g=(L-\lambda)(\xi u)=\xi f+2a^{ij}D_{i}uD_{j}\xi+u(L-c-\lambda)\xi.$$
By weighted Sobolev embedding theorem, see \cite[Theorem 1.3]{ecr}, we have
\begin{equation*}
\zeta u\in W_{w}^{1,p(1)}(\cD)
\end{equation*}
for any $\zeta\in C_{0}^{\infty}(B_{2R}(z))$. Hence, $g\in L_{w}^{p(1)}(\cD)$. By the choice of $\lambda$, the equation
$$(L-\lambda)v=g$$
has a solution in $\mathring{W}_{w}^{2,p(1)}(\cD)\subset\mathring{W}_{w}^{2,p}(\cD)$ which is unique in $\mathring{W}_{w}^{2,p}(\cD)$. We thus obtain that for $j=1$,
\begin{equation}\label{case j=1}
v=\xi u\in \mathring{W}_{w}^{2,p(1)}(\cD),\quad\forall~\xi\in C_{0}^{\infty}(B_{2R}(z)).
\end{equation}
If $p(1)<q$, then by repeating this argument with $p(1)$ in place of $p$, we get \eqref{case j=1} for $j=2$. In this way we conclude \eqref{regularity xi u}.

Next we prove \eqref{es u W2q}. By the choice of $\lambda$, for $j\geq1$ and any $\xi,\eta\in C_{0}^{\infty}(B_{2R}(z))$ such that $\eta=1$ on the support of $\xi$, we have
\begin{align*}
\|\xi u\|_{W_{w}^{2,p(j)}(\cD)}&\leq C\|\xi f+2a^{ij}D_{i}uD_{j}\xi+u(L-c-\lambda)\xi\|_{L_{w}^{p(j)}(\cD)}\\
&\leq C\big(\|f\|_{L_{w}^{q}(\cD_{2R})}+\|\eta u\|_{W_{w}^{1,p(j)}(\cD)}\big)\\
&\leq C\big(\|f\|_{L_{w}^{q}(\cD_{2R})}+\|\eta u\|_{W_{w}^{2,p(j-1)}(\cD)}\big),
\end{align*}
where we used the weighted Sobolev embedding theorem in the last inequality. By iterating the above inequality, we obtain that for any $\xi\in C_{0}^{\infty}(B_{3R/2}(z))$, there is an $\eta\in C_{0}^{\infty}(B_{7R/4}(z))$ such that
$$\|\xi u\|_{W_{w}^{2,q}(\cD)}\leq C\big(\|f\|_{L_{w}^{q}(\cD_{2R})}+\|\eta u\|_{W_{w}^{2,p}(\cD)}\big).$$
Finally, recalling the conclusion for the case when $p=q$, we have
$$\|\eta u\|_{W_{w}^{2,p}(\cD)}\leq C\|u\|_{W_{w}^{2,p}(\cD_{7R/4})}\leq C\big(\|f\|_{L_{w}^{p}(\cD_{7R/2})}+\| u\|_{L_{w}^{p}(\cD_{7R/2})}\big).$$
This yields \eqref{es u W2q} and the lemma is proved.
\end{proof}


Next we recall the resolvent operator of $L-\lambda I$ by
$$\cR_{\lambda}: L_{w}^{p}(\cD)\rightarrow \mathring{W}_{w}^{2,p}(\cD).$$
Then $\cR_{\lambda}$ is a bounded operator for $\lambda\geq\lambda_{1}$. 
The following properties of $\cR_{\lambda}$ for $\lambda$ large play an important role in proving Lemma \ref{sov es weight LP} below.

\begin{lemma}\label{prop cR}
Let the coefficients of $L$ be infinitely differentiable, $L1\leq0$, and $\lambda\geq\lambda_{1}$. Then
\begin{enumerate}
\item
For any bounded $f$ and any $\gamma\in (0,1)$, we have $\cR_{\lambda}f\in C^{1+\gamma}(\cD)$, $\cR_{\lambda}f=0$ on $\partial\cD$, and in $\cD$,
\begin{equation}\label{ine cR f}
|\cR_{\lambda}f(x)|\leq\cR_{\lambda}|f|(x)\leq\lambda^{-1}\sup_{x\in\cD}|f(x)|.
\end{equation}
\item
There exists an integer $m_{0}=m_{0}(n,p,\delta,\Lambda,\cD,[w]_{A_{p}},r_{0})$, such that for any $f\in L_{w}^{p}(\cD)$, we have
\begin{equation}\label{es cR}
\sup_{x\in\cD}|\cR_{\lambda_{1}}^{m_{0}}f(x)|\leq C\|f\|_{L_{w}^{p}(\cD)},
\end{equation}
where $C=C(n,p,\delta,\Lambda,\cD,[w]_{A_{p}},r_{0})$.
\end{enumerate}
\end{lemma}

\begin{proof}
For $f\in L^{\infty}(\cD)$, \eqref{ine cR f} is proved in \cite[Theorem 11.2.1(3)]{k}.
To prove \eqref{es cR}, we set $\alpha={n}/(n-1)$, $p(j)=\alpha^{j}p$, and
$$u_{0}=f,\quad u_{j}=\cR_{\lambda_{1}}^{j}f,\quad j\geq1.$$
Notice that for $j\geq1$, we have
$$\lambda_{1}u_{j+1}-Lu_{j+1}=u_{j}.$$
Therefore, by using the solvability and estimates for $\lambda$ large, we have $u_{j+1}\in\mathring{W}_{w}^{2,p}(\cD)$ and
$$\|u_{j+1}\|_{W_{w}^{2,p}(\cD)}\leq C\|u_{j}\|_{L_{w}^{p}(\cD)}\leq C\|u_{j}\|_{L_{w}^{p(j)}(\cD)}.$$
By using Lemma \ref{appen lem regularity}, we get
$$\|u_{j+1}\|_{W_{w}^{2,p(j)}(\cD)}\leq C\big(\|u_{j}\|_{L_{w}^{p(j)}(\cD)}+\|u_{j+1}\|_{L_{w}^{p}(\cD)}\big).$$
Hence,
\begin{equation}\label{estimates uj+1}
\|u_{j+1}\|_{W_{w}^{2,p(j)}(\cD)}\leq C\|u_{j}\|_{L_{w}^{p(j)}(\cD)}.
\end{equation}
By the weighted embedding theorem, we have
$$\|u_{j+1}\|_{L_{w}^{p(j+1)}(\cD)}\leq C\|u_{j}\|_{L_{w}^{p(j)}(\cD)}.$$
Iterating the above inequality yields that for $j\geq0$, we obtain
\begin{equation}\label{estimates uj+1 f}
\|u_{j}\|_{L_{w}^{p(j)}(\cD)}\leq C\|u_{0}\|_{L_{w}^{p(0)}(\cD)}=C\|f\|_{L_{w}^{p}(\cD)}.
\end{equation}
It follow from H\"{o}lder's inequality and the definition of $A_{p}$ weights that $u_{j+1}\in W^{2,p(j)/p}(\cD)$. Then we fix a $j=j(n,p)$ by choosing $p(j)>np/2$. For such $j$, we conclude from \eqref{estimates uj+1} and \eqref{estimates uj+1 f} that
$$\sup_{x\in\cD}|u_{j+1}(x)|\leq C\|u_{j+1}\|_{W^{2,p(j)/p}(\cD)}\leq C\|u_{j+1}\|_{W_{w}^{2,p(j)}(\cD)}\leq C\|f\|_{L_{w}^{p}(\cD)},$$
which shows that \eqref{es cR} holds with $m_{0}=j+1$. The lemma is proved.
\end{proof}

Next we show the solvability when $\lambda\ge \epsilon_0$ for a positive constant $\epsilon_{0}>0$.

\begin{lemma}\label{sov es weight LP}
Let $p\in (1,\infty)$, $w\in A_p$, $\epsilon_{0}>0$, and $L1\leq0$. Under the above conditions, for any $\lambda\geq\epsilon_{0}$ and $u\in \mathring{W}_{w}^{2,p}(\cD)$, we have
\begin{equation}\label{es Lp weight}
\|u\|_{W_{w}^{2,p}(\cD)}\leq C\|\lambda u-Lu\|_{L_{w}^{p}(\cD)},
\end{equation}
where $C$ depends on $n,p,\delta,\Lambda,\cD,\epsilon_0,[w]_{A_{p}}$, and $r_{0}$.
\end{lemma}

\begin{proof}
As noted after Theorem \ref{solvability weight Lp}, it suffices to prove the case when $\epsilon_{0}\leq\lambda<\lambda_{1}$.
We follow the idea in \cite[Section 11.3]{k}. Here we list the main differences. 
By approximations, we may assume that the coefficients are smooth. Similar to the proof of \cite[Theorem 8.5.6]{k}, we have
$$\|u\|_{W_{w}^{2,p}(\cD)}\leq C\big(\|\lambda u-Lu\|_{L_{w}^{p}(\cD)}+\|u\|_{L_{w}^{p}(\cD)}\big).$$
Therefore, it suffices to prove for $\varepsilon_{0}\leq\lambda<\lambda_{1}$, we have
$$\|u\|_{L_{w}^{p}(\cD)}\leq C\|\lambda u-Lu\|_{L_{w}^{p}(\cD)}.$$
For this, we define
$$f:=\lambda u-Lu.$$
Then
$$\lambda_{1}u-Lu=(\lambda_{1}-\lambda)u+f,\quad u=(\lambda_{1}-\lambda)\cR_{\lambda_{1}}u+\cR_{\lambda_{1}}f.$$
By induction on $m$, we have
\begin{equation}\label{induction m}
u=\big((\lambda_{1}-\lambda)\cR_{\lambda_{1}}\big)^{m}u+\sum_{i=0}^{m-1}\big((\lambda_{1}-\lambda)\cR_{\lambda_{1}}\big)^{i}\cR_{\lambda_{1}}f,
\end{equation}
where $m\geq1$ is any integer. We next introduce constants $C_{1}$ and $M_{m}$ such that
\begin{equation}\label{C1 Mm}
\|\cR_{\lambda_{1}}g\|_{L_{w}^{p}(\cD)}\leq C_{1}\|g\|_{L_{w}^{p}(\cD)}\quad\forall~g\in L_{w}^{p}(\cD),\quad M_{m}=\sum_{i=0}^{m-1}(\lambda_{1}-\epsilon_{0})^{i}C_{1}^{i+1}.
\end{equation}
By using \eqref{induction m}, $0<\lambda_{1}-\lambda\leq\lambda_{1}-\epsilon_{0}$, \eqref{C1 Mm}, \eqref{ine cR f}, and \eqref{es cR}, we obtain for $m>m_{0}$,
\begin{align*}
\|u\|_{L_{w}^{p}(\cD)}&\leq w(\cD)^{1/p}(\lambda_{1}-\epsilon_{0})^{m}\sup_{x\in\cD}|\cR_{\lambda_{1}}^{m-m_{0}}\cR_{\lambda_{1}}^{m_{0}}u(x)|+M_{m}\|f\|_{L_{w}^{p}(\cD)}\\
&\leq w(\cD)^{1/p}\lambda_{1}^{m_{0}}(1-\epsilon_{0}/\lambda_{1})^{m}\sup_{x\in\cD}|\cR_{\lambda_{1}}^{m_{0}}u(x)|+M_{m}\|f\|_{L_{w}^{p}(\cD)}\\
&\leq C_{2}w(\cD)^{1/p}\lambda_{1}^{m_{0}}(1-\epsilon_{0}/\lambda_{1})^{m}\|u\|_{L_{w}^{p}(\cD)}+M_{m}\|f\|_{L_{w}^{p}(\cD)}.
\end{align*}
Fixing $m>m_{0}$ such that
$$C_{2}w(\cD)^{1/p}\lambda_{1}^{m_{0}}(1-\epsilon_{0}/\lambda_{1})^{m}\leq1/2,$$
we get \eqref{es Lp weight}.
The lemma is proved.
\end{proof}

We now give

\begin{proof}[\bf Proof of Theorem \ref{solvability weight Lp}.]
By using the method of continuity, we only need to prove that for any $u\in \mathring{W}_{w}^{2,p}(\cD)$ and $\lambda\geq0$, we have \eqref{es Lp weight}.
To this end, without loss of generality we may assume that $\cD\subset B_{2R_{0}}$, where $R_{0}=\mbox{diam}~{\cD}$. We take the global barrier $v_{0}$ from \cite[Lemma 11.1.2]{k}:
$$v_{0}(x)=\cosh(4c_{0}R_{0})-\cosh(c_{0}|x|)$$
satisfies $Lv_{0}\leq-1$, and $v_{0}>0$ in $B_{4R_{0}}$ and $v_{0}=0$ on $\partial B_{4R_{0}}$, where $c_{0}>0$ is a constant to be chosen. Next we introduce a new operator $L'$ by
$$L'u=v_{0}^{-1}L(v_{0}u).$$
Notice that in $\cD\subset B_{2R_{0}}$, according to the construction of $v_{0}$, we have $L'1\leq-\delta'$ for a constant $\delta'>0$ depending on $n,\delta,\Lambda$, and $R_{0}$, provided that $c_0=c_0(d,\delta,\Lambda)$ is sufficiently large. By using Lemma \ref{sov es weight LP} applied to $L'':=L'+\delta'$, we get
\begin{align*}
\|u\|_{W_{w}^{2,p}(\cD)}&\leq C\|uv_{0}^{-1}\|_{W_{w}^{2,p}(\cD)}\\
&\leq C\|(\lambda+\delta')uv_{0}^{-1}-L''(uv_{0}^{-1})\|_{L_{w}^{p}(\cD)}\\
&= C\|(\lambda u-v_{0}L'(uv_{0}^{-1}))v_{0}^{-1}\|_{L_{w}^{p}(\cD)}\\
&=C\|(\lambda u-Lu)v_{0}^{-1}\|_{L_{w}^{p}(\cD)}\leq C\|(\lambda -L)u\|_{L_{w}^{p}(\cD)}.
\end{align*}
Hence, we finish the proof of the theorem.
\end{proof}

Finally we give the solvability of the adjoint operator of $L$ defined by
$$L^{*}u:=D_{ij}(a^{ij}u)-D_{i}(b^{i}u)+cu.$$
By using a similar argument in the proof of Lemma \ref{sol adjoint}, from Theorem \ref{solvability weight Lp}, we have
\begin{corollary}\label{solvability adjoint}
Let $p\in(1,\infty)$, $w\in A_{p}$, and $L1\leq0$. Assume that $g=(g^{ij})_{i,j=1}^{n}\in L_{w}^{p}(\cD)$. The coefficients $a^{ij}$, $b^{i}$, and $c$ satisfy the same conditions as imposed in Theorem \ref{solvability weight Lp}. Then for any $\lambda\geq0$,
\begin{align*}
\begin{cases}
L^{*}u-\lambda u=\Div^{2}g&\quad\mbox{in}~\cD,\\
u=\frac{g\nu\cdot\nu}{A\nu\cdot\nu}&\quad\mbox{on}~\partial\cD
\end{cases}
\end{align*}
admits a unique adjoint solution $u\in L_{w}^{p}(\cD)$. Moreover, the following estimate holds
\begin{equation}\label{adjoint es u}
\|u\|_{L_{w}^{p}(\cD)}\leq C\|g\|_{L_{w}^{p}(\cD)},
\end{equation}
where $C=C(n,p,\delta,\Lambda,\cD,r_{0},[w]_{A_{p}})$.
\end{corollary}

\section*{Acknowledgement} This work was completed while the second author was visiting Brown University. She would like to thank the Division of Applied Mathematics at Brown University for the hospitality and the stimulating environment. The authors would like to thank Prof. Seick Kim and Yanyan Li for helpful discussions.

\end{document}